\RequirePackage{fix-cm}
\documentclass[oneside,english]{amsart}
\usepackage[latin9]{inputenc}
\usepackage{geometry}
\usepackage{babel}
\usepackage{amsthm}
\usepackage{amstext}
\usepackage{amssymb}
\usepackage{fixltx2e}
\usepackage[unicode=true,
 bookmarks=true,bookmarksnumbered=false,bookmarksopen=false,
 breaklinks=false,pdfborder={0 0 1},backref=false,colorlinks=false]
 {hyperref}
\hypersetup{
 pdfauthor={Bachar Alhajjar}}

\makeatletter
\numberwithin{equation}{section}
\numberwithin{figure}{section}
\@ifundefined{lettrine}{\usepackage{lettrine}}{}
\theoremstyle{plain}
\newtheorem{thm}{\protect\theoremname}[section]
  \theoremstyle{definition}
  \newtheorem{defn}[thm]{\protect\definitionname}
  \theoremstyle{plain}
  \newtheorem{prop}[thm]{\protect\propositionname}
  \theoremstyle{remark}
  \newtheorem{rem}[thm]{\protect\remarkname}
  \theoremstyle{plain}
  \newtheorem{cor}[thm]{\protect\corollaryname}
  \theoremstyle{plain}
  \newtheorem{lem}[thm]{\protect\lemmaname}

\theoremstyle{definition}

\makeatother

  \providecommand{\corollaryname}{Corollary}
  \providecommand{\definitionname}{Definition}
  \providecommand{\lemmaname}{Lemma}
  \providecommand{\propositionname}{Proposition}
  \providecommand{\remarkname}{Remark}
\providecommand{\theoremname}{Theorem}

\begin{document}

\title{Methods to compute Ring Invariants {\tiny and} \\
Applications: a New class of Exotic Threefolds}

\author{Bachar~~ALHAJJAR }
\begin{abstract}
We develop some methods to compute the Makar-Limanov and Derksen invariants,
isomorphism classes and automorphism groups for $\mathbf{k}$-domains
$B$, which are constructed from certain Russell $\mathbf{k}$-domains.
We propose tools and techniques to distinguish between $\mathbf{k}$-domains
with the same Makar-Limanov and Derksen invariants. In particular,
we introduce the exponential chain associated to certain modifications.
We extract $\mathbb{C}$-domains from the class $B$ that have smooth
contractible factorial $\mathrm{Spec}(B)$, which are diffeomorphic
to $\mathbb{R}^{6}$ but not isomorphic to $\mathbb{C}^{3}$, that
is, exotic $\mathbb{C}^{3}$. We examine associated exponential chains
to prove that exotic threefolds $\mathrm{Spec}(B)$ are not isomorphic
to $\mathrm{Spec}(R)$, for any Russell $\mathbb{C}$-domain $R$.
\end{abstract}

\address{Department of Mathematics, Faculty of Science, Al-Furat University,
Deir ez-Zor, Syria.}

\curraddr{Institut de Math\'ematiques de Bourgogne, Universit\'e de Bourgogne,
Dijon, France.}

\email{Bachar.Alhajjar@gmail.com}


\maketitle

\section*{\textbf{\normalsize Introduction}}

This paper discusses some methods to compute Makar-Limanov and Derksen
invariants, isomorphism classes and automorphism groups of $\mathbf{k}$-domains.
It also proposes some techniques to distinguish between $\mathbf{k}$-domains
with the same Makar-Limanov and Derksen invariants.

Let $\mathbf{k}$ be a field of characteristic zero and let $A$ be
a commutative $\mathbf{k}$-domain. A $\mathbf{k}$-derivation $\partial\in\mathrm{Der}_{\mathbf{k}}(A)$
is said to be\emph{ locally nilpotent} if for every $a\in A$, there
is an integer $n\geq0$ such that $\partial^{n}(a)=0$. The \emph{Makar-Limanov
invariant} $\mathrm{ML}(A)$ is defined by L. Makar-Limanov as the
intersection of the kernels of all locally nilpotent derivations of
$A$. The\emph{ Derksen invariant} $\mathcal{D}(A)$ is defined by
H. Derksen to be the sub-algebra generated by the kernels of all non-zero
locally nilpotent derivations of $A$. The Makar-Limanov and Derksen
invariants are among the more important tools, arising from the study
of locally nilpotent derivations, due to their applications in distinguishing
between $\mathbf{k}$-domain and in studying isomorphism classes and
automorphism groups of $\mathbf{k}$-domain, see e.g. \cite{Kaliman Makar-Limanovr: AK invariant,Kaliman Makar-Limanov Threefolds,M-L1996,M-L2005,Makar-Limanov: A new ring invariant,Kaliman Venereau Zaidenberg,Crachiola,Derksen}.

We improve some techniques used in \cite{Alhajjar} to compute the
Makar-Limanov and Derksen invariants for certain $\mathbf{k}$-domains
of the form 
\[
B\simeq\mathbf{k}[X,Y,Z,T]/\langle X^{n}Y-(Y^{m}-X^{e}Z)^{d}-T^{r}-X\, Q(X,Y^{m}-X^{e}Z,T)\rangle.
\]
In \cite{Kaliman Makar-Limanovr: AK invariant}, S. Kaliman and L.
Makar-Limanov developed general techniques to determine the $\mathrm{ML}$-invariant
for a class of $\mathbf{k}$-domains $B=\mathbf{k}[X_{1},\ldots,X_{n}]/\mathfrak{b}$.
The idea, referred to as the homogeneization technique, is to reduce
the problem to the study of homogeneous locally nilpotent derivations
on graded algebras $\mathrm{Gr}(B)$ associated to $B$. For this,
one considers suitable filtrations $\mathcal{F}=\{\mathcal{F}_{i}\}_{i\in\mathbb{R}}$
on $B$ generated by $\mathbb{R}$-weight degree functions $\omega$
on $\mathbf{k}[X_{1},\ldots,X_{n}]$, in such a way that every non-zero
locally nilpotent derivation on $B$ induces a non-zero homogeneous
locally nilpotent derivation on the associated graded algebra $\mathrm{Gr_{\mathcal{F}}}(B)$.
The homogeneization technique is efficient when dealing with filtrations
that are proper, especially filtrations induced by $\mathbb{R}$-weight
degree functions $\omega$, which are appropriate for the ideal $\mathfrak{b}$.
Therefore, one surveys weights $\omega(X_{i})\in\mathbb{R}$ ; $i\in\{1,\ldots,n\}$,
which guarantee that the ideal $\widehat{\mathfrak{b}}$, generated
by top homogenous components of all elements in $\mathfrak{b}$, is
prime. We consider a different approach to achieve proper filtrations,
that is, we investigate weight degree functions on $\mathbf{k}[X_{1},\ldots,X_{n},Y_{n+1},\ldots,Y_{N}]=\mathbf{k}^{[N]}$
for certain choices of $N\in\mathbb{N}$ together with ideals $\mathfrak{a}\subset\mathbf{k}^{[N]}$
such that $B\simeq\mathbf{k}^{[N]}/\mathfrak{a}$ and the ideal $\widehat{\mathfrak{a}}$
is prime.

In a way similar to the one used in \cite{Alhajjar}, we construct
the new class from certain Russell $\mathbf{k}$-domains as follows.
Given two Russell $\mathbf{k}$-domains $R_{i}=\mathbf{k}[x,s,t,y_{i}]\simeq\mathbf{k}[X,Y_{i},S,T]/\langle X^{n_{i}}Y_{i}-F_{i}(X,S,T)\rangle$
for $i\in\{1,2\}$, via the localization homomorphism with respect
to $x$, we have $R_{1},R_{2}\subset\mathbf{k}[x,x^{-1},s,t]$, where
$y_{i}=x^{-n_{i}}F_{i}(x,s,t)$. The sub-algebra of $\mathbf{k}[x,x^{-1},s,t]$
generated by $R_{1}$ and $R_{2}$ coincides with $B:=R_{1}.R_{2}$
the sub-algebra of $\mathbf{k}[x,x^{-1},s,t]$ consists of all finite
sums of elements $ab$ where $a\in R_{1}$ and $b\in R_{2}$. That
is, $B=\mathbf{k}[x,s,t,y_{1},y_{2}]\simeq\mathbf{k}[X,Y_{1},Y_{2},S,T]/\mathfrak{b}$
for some prime ideal $\mathfrak{b}\subset\mathbf{k}[X,Y_{1},Y_{2},S,T]$,
which clearly contains the ideal $\langle X^{n_{1}}Y_{1}-F_{1}(X,S,T),X^{n_{2}}Y_{2}-F_{2}(X,S,T)\rangle$.
We show that $\mathcal{D}(B)=\mathbf{k}[x,s,t]$ and $\mathrm{ML}(B)=\mathbf{k}[x]$. 

We introduce contraction and exponential chains associated to \emph{exponential
modifications}, that is, modifications of $\mathbf{k}$-domains $A$
with locus $(a^{n},I)$, where $a\in A$ is an irreducible element,
$I$ is an ideal in $A$ and $a^{n}\in I$. An exponential modification
$A[I/a^{n}]$ has the chain $A[I/a^{n}]=\langle1\rangle\supset\langle a\rangle\supset\langle a^{2}\rangle\supset\cdots\supset\langle a^{n}\rangle$
of principal ideals in $A[I/a^{n}]$, which induces the chain $A=\langle1\rangle^{\mathbf{c}}\supset\langle a\rangle^{\mathbf{c}}\supset\langle a^{2}\rangle^{\mathbf{c}}\supset\cdots\supset\langle a^{n}\rangle^{\mathbf{c}}$
of ideals in $A$, that we call the \emph{contraction chain}, where
$\langle a^{N}\rangle^{\mathbf{c}}=\langle a^{N}\rangle\cap A$ is
the contraction of the ideal $\langle a^{N}\rangle\subset A[I/a^{n}]$
with respect to the inclusion $A\hookrightarrow A[I/a^{n}]$. In turn,
the contraction chain give rise to the chain $A\subset A[\langle a\rangle^{\mathbf{c}}/a]\subset A[\langle a^{2}\rangle^{\mathbf{c}}/a^{2}]\subset\cdots\subset A[\langle a^{n}\rangle^{\mathbf{c}}/a^{n}]=A[I/a^{n}]$
of sub-algebras of $A[I/a^{n}]$, which we call the\emph{ exponential
chain} of $A[I/a^{n}]$. 

In \cite{Alhajjar}, we introduced a family of ring invariants as
a generalization of the Derksen invariant. These invariants are certainly
useful to distinguish between $\mathbf{k}$-domains with the same
Derksen and Makar-Limanov invariants. In this paper we investigate
further techniques to distinguish between such $\mathbf{k}$-domains.
Certain conditions that two $\mathbf{k}$-domains, with the same Derksen
and Makar-Limanov invariants, must verify to be isomorphic can be
deduced from properties of their locally nilpotent derivations, see
section \ref{Sub:Basic-facts}. Also, for exponential modifications
with the same Derksen and Makar-Limanov invariants, necessary conditions
can be given by examining their associated exponential chains, see
Proposition \ref{Prop:n_1=00003Dn_2-and-e_1=00003De_2} and Theorem
\ref{Thm:The-new-algebras-are-noniso-to-the-Russell}. Indeed, a $\mathbf{k}$-isomorphism
$\Psi$ between exponential modifications $A[I/a^{n}]$ and $R$,
maps the exponential chain $A\subset A[\langle a\rangle^{\mathbf{c}}/a]\subset A[\langle a^{2}\rangle^{\mathbf{c}}/a^{2}]\subset\cdots\subset A[\langle a^{n}\rangle^{\mathbf{c}}/a^{n}]=A[I/a^{n}]$
isomorphically onto $\Psi(A)\subset\Psi(A)[\langle\Psi(a)\rangle^{\mathbf{c}}/\Psi(a)]\subset\cdots\subset\Psi(A)[\langle\Psi(a)^{n}\rangle^{\mathbf{c}}/\Psi(a)^{n}]=\Psi(A)[\Psi(I)/\Psi(a)^{n}]=R$.
In particular, if $a$ belongs to the Makar-Limanov invariant and
$A$ coincides with the Derksen invariant of the exponential modification
$A[I/a^{n}]$, then the exponential chain is invariant by any $\mathbf{k}$-automorphism
of $A[I/a^{n}]$, and by any locally nilpotent derivation of $A[I/a^{n}]$.
That is, $\psi(A[\langle a^{N}\rangle^{\mathbf{c}}/a^{N}])=A[\langle a^{N}\rangle^{\mathbf{c}}/a^{N}]$
and $\partial(A[\langle a^{N}\rangle^{\mathbf{c}}/a^{N}])\subset A[\langle a^{N}\rangle^{\mathbf{c}}/a^{N}]$
for every $\mathbf{k}$-automorphism $\psi$ of $A[I/a^{n}]$, every
locally nilpotent derivation $\partial$ of $A[I/a^{n}]$ and every
$N\in\{1,\ldots,n\}$.

We show that $\mathbf{k}$-domains $B$ of the new class of examples
arise as exponential modifications of the Derksen invariant $\mathbf{k}[x,s,t]$
with locus $(x^{n},I)$ for certain ideals $I\subset\mathbf{k}[x,s,t]$.
Also, we compute the contraction and exponential chains associated
to $B$ and we show that the exponential chain characterizes $B$,
then we proceed to determine isomorphism classes and automorphism
groups.

In the case $\mathbf{k}=\mathbb{C}$, we extract $\mathbb{C}$-domains
from the class $B$ that have smooth contractible factorial $\mathrm{Spec}(B)$,
which are diffeomorphic to $\mathbb{R}^{6}$ but not isomorphic to
$\mathbb{C}^{3}$, as their Makar-Limanov and Derksen invariants are
non-trivial, that is, exotic $\mathbb{C}^{3}$. These new exotic threefolds
$\mathrm{Spec}(B)$ are not isomorphic to $\mathrm{Spec}(R)$, for
any Russell $\mathbb{C}$-domain $R$. To show this we compare the
associated exponential chains. Indeed, the exponential chain of a
$\mathbf{k}$-domain $B$ (of the new class) has some identical members
while members of the exponential chain of a Russell $\mathbb{C}$-domains
are distinct from (even non-isomorphic to) each other.

\section{\textbf{Preliminaries}}

In this section we briefly recall basic facts in a form appropriate
to our needs, see \cite{Kaliman Makar-Limanovr: AK invariant,Zaidenberg}.
Unless otherwise specified $B$ will denote a commutative domain over
a field $\mathbf{k}$ of characteristic zero. The polynomial ring
in $n$ variables over the field $\mathbf{k}$ is denoted by $\mathbf{k}^{[n]}$.

\subsection{\label{Sub:Filtration-and-the} $\mathbb{Z}$-degree functions, $\mathbb{Z}$-filtrations
and associated graded algebras}
\begin{defn}
\label{Def:Z-degree-function} A $\mathbb{Z}$\emph{-degree function}
on $B$ is a map $\deg:B\longrightarrow\mathbb{Z}\,\mathbb{\cup}\{-\infty\}$
such that, for all $a,b\in B$, the following conditions hold: 

$(1)$ $\deg(a)=-\infty$ $\Leftrightarrow$ $a=0$.

$(2)$ $\deg(ab)=\deg(a)+\deg(b)$.

$(3)$ $\deg(a+b)\leq\max\{\deg(a),\deg(b)\}$.

\noindent  If the equality in (2) is replaced by the inequality $\deg(ab)\leq\deg(a)+\deg(b)$,
we say that $\deg$ is a $\mathbb{Z}$\emph{-semi-degree function}.
\end{defn}
\noindent  There is a one-to-one correspondence, see e.g. \cite{Zaidenberg,Crachiola Maubach},
between $\mathbb{Z}$-degree functions and proper $\mathbb{\mathbb{Z}}$-filtrations:
\begin{defn}
\label{Def:Z-proper-filtration} A \textit{$\mathbb{Z}$-filtration}
of $B$ is a collection $\{\mathcal{F}_{i}\}_{i\in\mathbb{Z}}$ of
sub-groups of $(B,+)$ with the following properties: 

1- $\mathcal{F}_{i}\subset\mathcal{F}_{i+1}$ for all $i\in\mathbb{Z}$. 

2- $B=\underset{i\in\mathbb{Z}}{\cup}\mathcal{F}_{i}$. 

3- $\mathcal{F}_{i}.\mathcal{F}_{j}\subset\mathcal{F}_{i+j}$ for
all $i,j\in\mathbb{Z}$. 

\noindent  The filtration is called \emph{proper} if the following
additional properties hold:

4- $\underset{i\in\mathbb{Z}}{\cap}\mathcal{F}_{i}=\{0\}$.

5- If $a\in\mathcal{F}_{i}\setminus\mathcal{F}_{i-1}$ and $b\in\mathcal{F}_{j}\setminus\mathcal{F}_{j-1}$,
then $ab\in\mathcal{F}_{i+j}\setminus\mathcal{F}_{i+j-1}$.
\end{defn}
\noindent  Indeed, for a $\mathbb{Z}$-degree function on $B$, the
sub-sets $\mathcal{F}_{i}=\{b\in B\mid\deg(b)\leq i\}$ are sub-groups
of $(B,+)$ that give rise to a proper $\mathbb{Z}$-filtration $\{\mathcal{F}_{i}\}_{i\in\mathbb{Z}}$.
Conversely, every proper $\mathbb{Z}$-filtration $\{\mathcal{F}_{i}\}_{i\in\mathbb{Z}}$,
yields a $\mathbb{Z}$-degree function $\omega:B\longrightarrow\mathbb{Z}\,\mathbb{\cup}\{-\infty\}$
defined by $\omega(0)=-\infty$ and $\omega(b)=i$ if $b\in\mathcal{F}_{i}\setminus\mathcal{F}_{i-1}$,
such an integer $i$ exists by property 4 of proper filtrations. 
\begin{defn}
\label{Def:Graded-gr} Given a $\mathbf{k}$-domain $B=\underset{i\in\mathbb{Z}}{\cup}\mathcal{F}_{i}$
equipped with a proper $\mathbb{Z}$-filtration $\mathcal{F}=\{\mathcal{F}_{i}\}_{i\in\mathbb{Z}}$,
the associated graded algebra $\mathrm{Gr}(B)$ is the abelian group
\[
\mathrm{Gr}(B):=\underset{i\in\mathbb{Z}}{\oplus}\mathcal{F}_{i}/\mathcal{F}_{i-1}
\]
equipped with the unique multiplicative structure for which the product
of the elements $a+\mathcal{F}_{i-1}\in\mathcal{F}_{i}/\mathcal{F}_{i-1}$
and $b+\mathcal{F}_{j-1}\in\mathcal{F}_{j}/\mathcal{F}_{j-1}$, where
$a\in\mathcal{F}_{i}$ and $b\in\mathcal{F}_{j}$, is the element
\[
(a+\mathcal{F}_{i-1})(b+\mathcal{F}_{j-1}):=ab+\mathcal{F}_{i+j-1}\in\mathcal{F}_{i+j}/\mathcal{F}_{i+j-1}.
\]
Property 5 for a proper filtration in Definition \ref{Def:Z-proper-filtration}
ensures that $\mathrm{Gr}(B)$ is a commutative $\mathbf{k}$-domain
when $B$ is an integral domain. Since for each $a\in B\setminus\{0\}$
the set $\{n\in\mathbb{Z}\mid a\in\mathcal{F}_{n}\}$ has a minimum
(by property 4 of proper filtrations), there exists $i$ such that
$a\in\mathcal{F}_{i}$ and $a\notin\mathcal{F}_{i-1}$. So we can
define a $\mathbf{k}$-linear map $\mathrm{gr}:B\longrightarrow\mathrm{Gr}(B)$
by sending $a$ to its class in $\mathcal{F}_{i}/\mathcal{F}_{i-1}$,
i.e $b\mapsto b+\mathcal{F}_{i-1}$, and $\mathrm{gr}(0)=0$. We will
frequently denote $\mathrm{gr}(b)$ simply by $\hat{b}$. Observe
that $\mathrm{gr}(b)=0$ if and only if $a=0$. 
\end{defn}

\subsection{The homogeneization technique}
\begin{defn}
By a \textit{$\mathbf{k}$-derivation} of $B$, we mean a $\mathbf{k}$-linear
map $D:B\longrightarrow B$ which satisfies the Leibniz rule: For
all $a,b\in B$; $D(ab)=aD(b)+bD(a)$. The \textit{kernel} of a derivation
$D$ is the subalgebra $\ker D=\left\{ b\in B;D(b)=0\right\} $ of
$B$. A $\mathbf{k}$-derivation $D\in\mathrm{Der}_{\mathbf{k}}(B)$
is said to be\emph{ locally nilpotent} if for every $a\in B$, there
exists $n\in\mathbb{Z}_{\geq0}$ $($depending of $a$$)$ such that
$\partial^{n}(a)=0$. The set of all locally nilpotent derivations
of $B$ is denoted by $\mathrm{\mathrm{LND}}(B)$.
\end{defn}
It is convenient to reduce the study of $\mathrm{\mathrm{LND}}(B)$
to the study of homogeneous locally nilpotent derivations on a graded
algebra $\mathrm{Gr_{\mathcal{F}}}(B)$, associated to a suitable
filtration $\mathcal{F}=\{\mathcal{F}_{i}\}_{i\in\mathbb{Z}}$ of
$B$, in such a way that every non-zero locally nilpotent derivation
on $B$ induces a non-zero homogeneous locally nilpotent derivation
on the associated graded algebra $\mathrm{Gr_{\mathcal{F}}}(B)$.
This technique, which is due to Makar-Limanov \cite{Makar-Limanov: A new ring invariant},
of replacing a locally nilpotent derivation by the induced homogeneous
one is called ``homogeneization of derivations'' or simply \emph{homogeneization}
technique, see \cite{D. Daigle: Tame and wild}.
\begin{defn}
\label{Def Graded} Given a $\mathbf{k}$-domain $B=\underset{i\in\mathbb{Z}}{\cup}\mathcal{F}_{i}$
equipped with a proper $\mathbb{Z}$-filtration, a $\mathbf{k}$-derivation
$D$ of $B$ is said to \emph{respect} the filtration if there exists
an integer $\tau$ such that $D(\mathcal{F}_{i})\subset\mathcal{F}_{i+\tau}$
for all $i\in\mathbb{Z}$. The smallest integer $\tau$, such that
$D(\mathcal{F}_{i})\subset\mathcal{F}_{i+\tau}$ for all $i\in\mathbb{Z}$,
is called the \emph{degree }of\emph{ }$D$ with respect to $\mathcal{F}=\{\mathcal{F}_{i}\}_{i\in\mathbb{Z}}$
and denoted by $\deg_{\mathcal{F}}D$.

\noindent  Note that if $D$ respects the filtration $\mathcal{F}=\{\mathcal{F}_{i}\}_{i\in\mathbb{Z}}$
then $\deg_{\mathcal{F}}D$ is well-defined. Indeed, denote by $\deg$
the $\mathbb{Z}$-degree function corresponding to $\mathcal{F}=\{\mathcal{F}_{i}\}_{i\in\mathbb{Z}}$
and let $U$ be the non-empty subset of $\mathbb{Z}\cup\{-\infty\}$
defined by $U:=\left\{ \deg\left(D(b)\right)-\deg\left(b\right)\,;\, b\in B\setminus\{0\}\right\} $.
Since $D$ respects the filtration $\mathcal{F}$, the set $U$ is
bounded above by $\tau_{0}$. Thus it has a greatest element $\tau$
which coincides with $\deg_{\mathcal{F}}D$ by definition. 

\noindent  Suppose that $D$ respects the filtration $\mathcal{F}=\{\mathcal{F}_{i}\}_{i\in\mathbb{Z}}$
and let $\tau=\deg_{\mathcal{F}}D$. We define a $\mathbf{k}$-linear
map $\widehat{D}:\mathrm{Gr}(B)\longrightarrow\mathrm{Gr}(B)$ as
follows: If $D=0$, then $\widehat{D}=0$ the zero map. Otherwise,
if $D\neq0$ then we define 
\[
\widehat{D}:\mathcal{F}_{i}/\mathcal{F}_{i-1}\longrightarrow\mathcal{F}_{i+\tau}/\mathcal{F}_{i+\tau-1}
\]
by the rule $\widehat{D}(a+\mathcal{F}_{i-1})=D(a)+\mathcal{F}_{i+\tau-1}$.
Now extend $\widehat{D}$ to all of $\mathrm{Gr}(B)$ by linearity.
One checks that $\widehat{D}$ satisfies the Leibniz rule, therefore
it is a homogeneous $\mathbf{k}$-derivation of $\mathrm{Gr}(B)$
of degree $\tau$, that is, $\widehat{D}$ sends homogeneous elements
of degree $i$ to either the zero element in $\mathrm{Gr}(B)$ or
to homogeneous elements of degree $i+\tau$ . Note that $\widehat{D}=0$
if and only if $D=0$, and that $\mathrm{gr}(\ker D)\subset\ker\widehat{D}$.
\end{defn}

\subsection{$\mathbb{Z}$-weight degree functions}

\indent\newline\noindent  Let $\mathfrak{b}$ be a prime ideal in
$\mathbf{k}^{[n]}$, in this paper we are interested in $\mathbb{Z}$-degree
functions $\deg$ on $\mathbf{k}^{[n]}/\mathfrak{b}$, which are induced
by $\mathbb{Z}$-weight degree functions on the polynomial algebra
$\mathbf{k}^{[n]}$. Degree functions $\deg$ that satisfy $\deg(\lambda)=0$
for every $\lambda\in\mathbf{k}\backslash\{0\}$ is referred to as
degree functions \emph{over} $\mathbf{k}$.
\begin{defn}
\label{Def:weight-appropriate} A $\mathbb{Z}$-\emph{weight degree
function} on the polynomial algebra $\mathbf{k}^{[n]}=\mathbf{k}[X_{1},\ldots,X_{n}]$
is a $\mathbb{Z}$-degree function (over $\mathbf{k}$) $\omega$
such that $\omega(P)=\max\{\omega(M)\,;\, M\in\mathcal{M}(P)\}$,
where $P\in\mathbf{k}^{[n]}$ is a non-zero polynomial, and $\mathcal{M}(P)$
is the set of non-zero monomials of $P$. Clearly, $\omega$ is uniquely
determined by the weights $\omega(X_{i})\in\mathbb{Z},\, i\in\{1,\ldots,n\}$.
A $\mathbb{Z}$-weight degree function $\omega$ defines a grading
$\mathbf{k}^{[n]}=\oplus_{l\in\mathbb{Z}}\mathbf{k}_{l}^{[n]}$ where
$\mathbf{k}_{l}^{[n]}\setminus\{0\}$ consists of all the $\omega$-homogeneous
polynomials of $\omega$-degree $l$. Accordingly, for any $P\in\mathbf{k}^{[n]}\backslash\{0\}$
we have a unique decomposition $P=P_{l_{1}}+\cdots+P_{l_{j}}$ into
a sum of $\omega$-homogeneous components $P_{l_{i}}$ of $\omega$-degree
$l_{i}$ where $l_{1}<l_{2}<\cdots<l_{j}=\omega(P)$. We call $\widehat{P}:=P_{l_{j}}$
the \emph{highest homogeneous component} of $P$ or the \emph{principal
component} of $P$. It is clear that $\widehat{PQ}=\widehat{P}\widehat{Q}$.
\end{defn}
Given a finitely generated $\mathbf{k}$-domain $B\simeq\mathbf{k}^{[n]}/\mathfrak{b}$
where $\mathfrak{b}$ is a prime ideal in $\mathbf{k}^{[n]}$, let
$\pi:\mathbf{k}^{[n]}\longrightarrow B$ be the natural projection.
Denote by $\widehat{\mathfrak{b}}$ the (graded) ideal in $\mathbf{k}^{[n]}$
generated by the highest homogeneous components of all elements of
$\mathfrak{b}$. 
\begin{defn}
\label{Def:appropiate} We say that a $\mathbb{Z}$-weight degree
function $\omega$ on $\mathbf{k}^{[n]}$ is \emph{appropriate for
an ideal} $\mathfrak{b}$ if the following conditions hold:

$($a$)$ $\mathfrak{b}\subset\langle X_{1},\ldots,X_{n}\rangle$.

$($b$)$ The ideal $\widehat{\mathfrak{b}}$ is prime and $X_{i}\notin\widehat{\mathfrak{b}}\,\,;\,\,\forall i=1,\ldots,n$.

\noindent  Assume that $\omega$ on $\mathbf{k}^{[n]}$ is appropriate
for the ideal $\mathfrak{b}$, for every non-zero $p\in B$ set 
\[
\omega_{B}(p):=\min_{P\in\pi^{-1}(p)}\omega(P).
\]
The next Proposition \ref{Prop:KeyPro-1-2-3}, which is due to Kaliman
and Makar-Limanov, ensures that $\omega_{B}$ is a $\mathbb{Z}$-degree
function on $B$. Therefore, the filtration $\mathcal{F}_{\omega_{B}}=\{\mathcal{F}_{i}\}_{i\in\mathbb{Z}}$
induced by $\omega_{B}$ is a proper $\mathbb{Z}$-filtration of $B\simeq\mathbf{k}[X_{1},\ldots,X_{n}]/\mathfrak{b}$.
Moreover, the proposition provides a description of the associated
graded algebra $\mathrm{Gr}(B)$. Finally, it asserts in particular
that every locally nilpotent derivation respects the proper filtration
$\mathcal{F}_{\omega_{B}}$.\end{defn}
\begin{prop}
\label{Prop:KeyPro-1-2-3} \cite[Lemma 3.2, Proposition 4.1, and Lemma 5.1]{Kaliman Makar-Limanovr: AK invariant}
Let $B=\mathbf{k}[x_{1},\ldots,x_{n}]\simeq\mathbf{k}[X_{1},\ldots,X_{n}]/\mathfrak{b}$
be a finitely generated $\mathbf{k}$-domain and let $\omega$ be
a $\mathbb{Z}$-weight degree function on $\mathbf{k}[X_{1},\ldots,X_{n}]$.
Suppose that $\omega$ is appropriate for the ideal $\mathfrak{b}$,
then:

$(1)$ $\omega_{B}$ is a $\mathbb{Z}$-degree function on $B$ and
$\omega_{B}(x_{i})=\omega(X_{i})\,;\, i=1,\ldots,n$.

$(2)$ The graded algebra $\mathrm{Gr}(B)$ associated to the proper
$\mathbb{Z}$-filtration $\mathcal{F}_{\omega_{B}}=\{\mathcal{F}_{i}\}_{i\in\mathbb{Z}}$
is isomorphic to $\mathbf{k}^{[n]}/\widehat{\mathfrak{b}}$.

$(3)$ Every derivation $\partial$ of $B$ respects the $\omega_{B}$-filtration
$\mathcal{F}_{\omega_{B}}=\{\mathcal{F}_{i}\}_{i\in\mathbb{Z}}$,
that is, there exists $\tau$ such that $\partial(F_{i})\subset F_{i+\tau}$
for every $i\in\mathbb{Z}$. Consequently, $\deg_{\omega_{B}}(\partial)<\infty$
and $\partial$ induces a derivation $\widehat{\partial}$ of $\mathrm{Gr}(B)$
which is locally nilpotent whenever $\partial$ is.
\end{prop}

\section{\textbf{\label{Sec:The-class-B_(n,e,Q)} A New Class of Examples}}

In this section, we consider a family of commutative finitely generated
$\mathbf{k}$-domains of the following form:{\large 
\[
B:=\mathbf{k}[x,y,z,t]\simeq\mathbf{k}[X,Y,Z,T]/\langle X^{n}Y-(Y^{m}-X^{e}Z)^{d}-T^{r}-X\, Q(X,Y^{m}-X^{e}Z,T)\rangle,
\]
}\\
where $e\geq0$, \emph{$n\geq1$} such that $(n,e)\neq(1,0)$, \emph{$m,d,r\geq2$
}such that \emph{$\gcd(d,r)=1$}, and $Q(X,S,T)\in\mathbf{k}[X,S,T]$.

\subsection{\label{Sub:Algebraic-construction:} Algebraic construction }

\indent\newline\noindent  Here, we explain how to construct the new
class $B$ from Russell $\mathbf{k}$-domains:
\begin{defn}
\label{Def:Russell-domains} Given an integer $n\in\mathbb{N}$ and
a polynomial $F(X,S,T)\in\mathbf{k}[X,S,T]$ such that $P(S,T):=F(0,S,T)\notin\mathbf{k}$,
we define the \emph{Russell} $\mathbf{k}$-domain corresponding to
the pair $(n,F)$ to be the $\mathbf{k}$-domain; 
\[
\mathbf{R}_{(n,F)}:=\mathbf{k}[x,y,s,t]\simeq\mathbf{k}[X,Y,S,T]/\langle X^{n}Y-F(X,S,T)\rangle.
\]

\end{defn}
Consider the Russell $\mathbf{k}$-domain $R=\mathbf{R}_{(n,S^{d}+T^{r}+X\, Q(X,S,T))}$
corresponding to the pair $(n,S^{d}+T^{r}+X\, Q(X,S,T))$. It is isomorphic
to $\mathbf{k}[X,Y,Z,T]/\langle X^{n}Y-(Y^{m}-Z)^{d}-T^{r}-X\, Q(X,Y^{m}-Z,T)\rangle$,
which is a member of the new family \ref{Sec:The-class-B_(n,e,Q)}
that corresponds to $e=0$.

Denote by $z$ the element $z:=x^{-nm-e}((s^{d}+t^{r}+x\, Q(x,s,t))^{m}-x^{nm}s)\in\mathbf{k}[x,x^{-1},s,t]$.
That is, $z$ is an algebraic element over $\mathbf{k}[x,s,t]$ that
has the following minimal polynomial; 
\[
x^{nm+e}Z-(s^{d}+t^{r}+x\, Q(x,s,t))^{m}+x^{nm}s\in\mathbf{k}[x,s,t][Z].
\]
 Thus, 
\[
\mathbf{k}[x,s,t,z]\simeq\mathbf{k}[X,S,T,Z]/\langle X^{nm+e}Z-(S^{d}+T^{r}+X\, Q(X,S,T))^{m}+X^{nm}S\rangle.
\]
The ring $\mathbf{k}[x,s,t,z]$ is the Russell $\mathbf{k}$-domain
corresponding to the pair $(nm+e,(S^{d}+T^{r}+X\, Q(X,S,T))^{m}-X^{nm}S)$.

Extend $R$ to $B:=R[z]$, the sub-algebra of $\mathbf{k}[x,x^{-1},s,t]$
generated by $z$ and $R\subset\mathbf{k}[x,x^{-1},s,t]$, where the
inclusion is induced be the localization homomorphism with respect
to $x$. Then, 
\[
B=\mathbf{k}[x,y,s,t,z]\simeq\mathbf{k}[X,Y,Z,S,T]/\langle X^{n}Y-S^{d}-T^{r}-X\, Q(X,S,T),Y^{m}-X^{e}Z-S\rangle.
\]
Hence,
\[
B\simeq\mathbf{k}[X,Y,Z,T]/\langle X^{n}Y-(Y^{m}-X^{e}Z)^{d}-T^{r}-X\, Q(X,Y^{m}-X^{e}Z,T)\rangle.
\]

Note that $R,\mathbf{k}[x,s,t,z]\subset B=R.\mathbf{k}[x,s,t,z]$,
where $R.\mathbf{k}[x,s,t,z]$ is by definition the sub-algebra of
$\mathbf{k}[x,x^{-1},s,t]$ consists of all finite sums of elements
$ab$ where $a\in R$ and $b\in\mathbf{k}[x,s,t,z]$. This simply
means that $B$ can be realized as the sub-algebra of $\mathbf{k}[x,x^{-1},s,t]$
generated by both $R$ and $\mathbf{k}[x,s,t,z]$.

\subsection{\textmd{$\mathbb{Z}\,$-}degree functions, \textmd{$\mathbb{Z}\,$-}filtrations,
and associated graded algebras}

\indent\newline\noindent  Given a $\mathbf{k}$-domain $A\simeq\mathbf{k}[X_{1},\ldots,X_{n}]/\mathfrak{a}$,
we consider proper $\mathbb{Z}$-filtrations on $A$ induced by $\mathbb{Z}$-weight
degree functions on $\mathbf{k}[X_{1},\ldots,X_{n},Y_{n+1},\ldots,Y_{N}]=\mathbf{k}^{[N]}$
for certain choices of $N\in\mathbb{N}$ together with ideals $\mathfrak{b}\subset\mathbf{k}^{[N]}$
such that the ring $A$ can be identified with $\mathbf{k}^{[N]}/\mathfrak{b}$
and the ideal $\widehat{\mathfrak{b}}$ is prime. We refer to this
technique as the twisted embedding technique, see \cite[Sub-section 2.2.2]{Alhajjar}.
It is also convenient to apply the homogeneization technique to proper
filtrations $\{\mathcal{F}_{i}\}_{i\in\mathbb{Z}}$ which give raise
to graded algebras with one dimensional graded pieces, that is, the
corresponding graded pieces $A_{[i]}:=\mathcal{F}_{i}/\mathcal{F}_{i-1}$
are generated by one element as $A_{[0]}$-modules. In particular,
this is the case for filtrations $\{\mathcal{F}_{i}\}_{i\in\mathbb{Z}}$
that satisfy the condition: for every $i\in\mathbb{Z}$, the $\mathcal{F}_{0}$-module
$\mathcal{F}_{i}$ is generated by $|i|+1$ element.

Note that the $\mathbf{k}$-domain 
\[
B:=\mathbf{k}[x,y,z,t]\simeq\mathbf{k}[X,Y,Z,T]/\langle X^{n}Y-(Y^{m}-X^{e}Z)^{d}-T^{r}-X\, Q(X,Y^{m}-X^{e}Z,T)\rangle
\]
is isomorphic to $\mathbf{k}[X,Y,Z,S,T]/\mathfrak{b}$, where $\mathfrak{b}$
the ideal in $\mathbf{k}^{[5]}=\mathbf{k}[X,Y,Z,S,T]$ defined by
\[
\mathfrak{b}=\langle X^{n}Y-S^{d}-T^{r}-X\, Q(X,S,T),\, Y^{m}-X^{e}Z-S\rangle.
\]
That is, $B=\mathbf{k}[x,y,z,t]=\mathbf{k}[x,s,t,y,z]\simeq\mathbf{k}^{[5]}/\mathfrak{b}$,
where $s=y^{m}-x^{e}z$.
\begin{defn}
\label{Def:weight-degree-on-B} Let $\omega$ be the $\mathbb{Z}$-weight
degree function on $\mathbf{k}^{[5]}$ defined by 
\[
\omega(X,Y,Z,S,T)=(-1,n,nm+e,0,0).
\]

\end{defn}
\noindent  Let $\widehat{\mathfrak{b}}$ be the ideal in $\mathbf{k}^{[5]}$
generated by highest homogeneous components, relative to $\omega$,
of all elements in $\mathfrak{b}$. The highest homogeneous components
of $X^{n}Y-S^{d}-T^{r}-X\, Q(X,S,T)$ and $Y^{m}-X^{e}Z-S$ are the
irreducible polynomials $X^{n}Y-S^{d}-T^{r}$ and $Y^{m}-X^{e}Z$
(respectively) in $\mathbf{k}^{[5]}$. Using properties of the graded
map $\mathrm{gr}_{\omega}:\mathbf{k}^{[5]}\longrightarrow\mathbf{k}^{[5]}$
presented in \cite[Lemma 1.4]{Alhajjar}, one checks that the ideal
$\widehat{\mathfrak{b}}$ coincides with $\langle X^{n}Y-S^{d}-T^{r},\, Y^{m}-X^{e}Z\rangle$,
see also Remark \ref{Rem:determine the graded ideal} below. Furthermore,
the ideal $\widehat{\mathfrak{b}}=\langle X^{n}Y-S^{d}-T^{r},\, Y^{m}-X^{e}Z\rangle$
is prime. Indeed, note that $\mathbf{k}^{[5]}/\widehat{\mathfrak{b}}\simeq\mathbf{R}_{(n,S^{d}+T^{r})}[Z]/\langle y^{m}-x^{e}Z\rangle$,
where $\mathbf{R}_{(n,S^{d}+T^{r})}=\mathbf{k}[x,s,t,y]\simeq\mathbf{k}[X,Y,S,T]/\langle X^{n}Y-S^{d}-T^{r}\rangle$
is the Russell $\mathbf{k}$-domain corresponding to the pair $(n,S^{d}+T^{r})$.
Since a polynomial of degree one $P(Z)=aZ+b\in\mathbf{R}_{(n,S^{d}+T^{r})}[Z]$
is irreducible if and only if $a$ and $b$ have no common factors
in $\mathbf{R}_{(n,S^{d}+T^{r})}$, we conclude that $y^{m}-x^{e}Z\in\mathbf{R}_{(n,S^{d}+T^{r})}[Z]$
is irreducible. Furthermore, by Gauss's Lemma, $y^{m}-x^{e}Z$ is
prime as $\mathbf{R}_{(n,S^{d}+T^{r})}[Z]$ is factorial (since $\mathbf{R}_{(n,S^{d}+T^{r})}$
is factorial by virtue of \cite[Lemma 1]{Nagata}). Therefore, we
deduce that $\mathbf{R}_{(n,S^{d}+T^{r})}[Z]/\langle y^{m}-x^{e}Z\rangle\simeq\mathbf{k}^{[5]}/\widehat{\mathfrak{b}}$
is a $\mathbf{k}$-domain and hence $\widehat{\mathfrak{b}}$ is prime.
\begin{rem}
\label{Rem:determine the graded ideal} Let $\mathfrak{a}=\langle P,Q\rangle$
be the ideal (not necessary prime) generated by elements $P,Q\in\mathbf{k}^{[N]}$,
and let $\omega$ be a weight degree on $\mathbf{k}^{[N]}$. Recently
Moser-Jauslin informed us that $\widehat{\mathfrak{a}}=\langle\hat{P},\hat{Q}\rangle$
whenever $\gcd(\hat{P},\hat{Q})=1$ and provided the following argument.
Given $H\in\mathfrak{b}$ there exist $f,g\in\mathbf{k}^{[N]}$ such
that $H=fP+gQ$. Note that the pair $(f,g)$ can be chosen such that
$\omega(fP)\leq\omega(H)$. Indeed, if not then for every such pair
$(f,g)$ we have $\omega(fP),\omega(gQ)>\omega(H)$. Thus $\omega(f)$
bounded below by $\omega(H)-\omega(P)$. So $f$ can be chosen to
be of minimal degree. On the other hand, condition $\omega(fP),\omega(gQ)>\omega(H)$
implies that $\omega(fP)=\omega(gQ)$ and $\hat{f}\hat{P}+\hat{g}\hat{Q}=0$,
see \cite[Lemma 1.4 (4)]{Alhajjar}. Since $\gcd(\hat{P},\hat{Q})=1$,
we conclude $\hat{Q}$ divides $\hat{f}$. Write $\hat{f}=u\hat{Q}$
and let $f_{0}=f-uQ$. Then we get $H=f_{0}P+(g+uP)Q$ with $\omega(f_{0})<\omega(f)$.
This contradicts the minimality of the degree of $f$. Therefore,
since $\omega(fP)\leq\omega(H)$, we conclude that $\hat{H}$ is either
$\hat{g}\hat{Q}$ (if $\omega(fP)<\omega(gQ)$), see \cite[Lemma 1.4 (2)]{Alhajjar},
or $\hat{f}\hat{P}+\hat{g}\hat{Q}$ (if $\omega(fP)=\omega(gQ)$),
see \cite[Lemma 1.4 (3)]{Alhajjar}. Hence, $\widehat{\mathfrak{a}}=\langle\hat{P},\hat{Q}\rangle$.
\end{rem}
\noindent  Thus, we conclude that $\omega$ is appropriate for the
ideal $\mathfrak{b}$ and hence $\omega$ induces $\omega_{B}$ a
$\mathbb{Z}$-degree function on $B$, see Proposition \ref{Prop:KeyPro-1-2-3}
(1), where
\[
\omega_{B}(p):=\min_{P\in\pi^{-1}(p)}\{\omega(P)\}.
\]
Noting that the proper $\mathbb{Z}$-filtration of $\mathbf{k}[X,Y,Z,S,T]$
induced by $\omega$ is given by 
\[
\Im_{\alpha}=\underset{\alpha=(nm+e)i+nj-l}{\oplus}\mathbf{k}[S,T].X^{l}Y^{j}Z^{i}\oplus\Im_{\alpha-1}\,;\,\textrm{ \ensuremath{i,j,l\in\mathbb{Z},}}
\]
we obtain the following.
\begin{prop}
\emph{\label{Prop: the filtration new example}} Let $\mathcal{F}=\{\mathcal{F}_{\alpha}=\pi(\Im_{\alpha})\}_{\alpha\in\mathbb{Z}}$
be the proper $\mathbb{Z}$-filtration on $B$ induced by $\omega_{B}$.
Then:

$(1)$ $\mathcal{F}_{-i}=\mathbf{k}[s,t]x^{i}+\mathcal{F}_{-i-1}$
for every $i>0$,

$(2)$ $\mathcal{F}_{0}=\mathbf{k}[s,t]+\mathcal{F}_{-1}=\mathbf{k}[x,s,t]$,

$(3)$ $\mathcal{F}_{nj-l}=\mathbf{k}[s,t]x^{l}y^{j}+\mathcal{F}_{nj-l-1}$
for $l\in\{0,\ldots,n-1\}$ and $j\in\mathbb{N}$,

$(4)$ $\mathcal{F}_{(nm+e)i-l}=\mathbf{k}[s,t]x^{l}z^{i}+\mathcal{F}_{(nm+e)i-l-1}$
for $l\in\{0,\ldots,e-1\}$ and $i\in\mathbb{N}$,

$(5)$ $\mathcal{F}_{(nm+e)i+nj-l}=\mathbf{k}[s,t]x^{l}y^{j}z^{i}+\mathcal{F}_{(nm+e)i+nj-l-1}$
for $l\in\{0,\ldots,\min\{n,e\}-1\}$ and $i,j\in\mathbb{N}\backslash\{0\}$.\end{prop}
\begin{cor}
\emph{\label{Cor: the grading new example} }The graded algebra $\mathrm{Gr}(B)$
associated to $\mathcal{F}=\{\mathcal{F}_{\alpha}=\pi(\Im_{\alpha})\}_{\alpha\in\mathbb{Z}}$
is isomorphic to 
\[
\mathrm{Gr}(B)\simeq\mathbf{k}[X,Y,Z,S,T]/\widehat{\mathfrak{b}}=\mathbf{k}[X,Y,Z,S,T]/\langle X^{n}Y-S^{d}-T^{r},\, Y^{m}-X^{e}Z\rangle.
\]
\emph{ }Furthermore, denote by $B_{[i]}=\mathcal{F}_{i}/\mathcal{F}_{i-1}$.
Then:

$(1)$ $B_{[-i]}=\mathbf{k}[\widehat{s},\widehat{t}]\widehat{x}^{i}$
for $i>0$, 

$(2)$ $B_{[0]}=\mathbf{k}[\widehat{s},\widehat{t}]$,

$(3)$ $B_{[nj-l]}=\mathbf{k}[\widehat{s},\widehat{t}]\widehat{x}^{l}\widehat{y}^{j}$
for $l\in\{0,\ldots,n-1\}$ and $j\in\{0,\ldots,m-1\}$,

$(4)$ $B_{[(nm+e)i-l]}=\mathbf{k}[\widehat{s},\widehat{t}]\widehat{x}^{l}\widehat{z}^{i}$
for $l\in\{0,\ldots,e-1\}$ and $i\in\mathbb{N}$,

$(5)$ $B_{[(nm+e)i+nj-l]}=\mathbf{k}[\widehat{s},\widehat{t}]\widehat{x}^{l}\widehat{y}^{j}\widehat{z}^{i}$,
for $l\in\{0,\ldots,\min\{n,e\}-1\}$ and $i,j\in\mathbb{N}\backslash\{0\}$.
\end{cor}

\subsection{The Derksen invariant and degree of derivations}

\indent\newline\noindent  Recall that the \emph{Derksen invariant}
of a $\mathbf{k}$-domain $A$ is defined to be the sub-algebra $\mathcal{D}(A)\subset A$
generated by the kernels of all non-zero locally nilpotent derivation
of $A$. That is, $\mathcal{D}(A):=\mathbf{k}[\cup_{\partial\in\mathrm{\mathrm{LND}}(A)\backslash\{0\}}\ker\partial]\subset A$.
The following theorem determines the Derksen invariant for the class
\ref{Sec:The-class-B_(n,e,Q)}.
\begin{thm}
\label{Thm:Derkson-invariant-of B} Let $B$ be the $\mathbf{k}$-domain
defined by $B:=\mathbf{k}[x,y,z,t]\simeq\mathbf{k}[X,Y,Z,T]/\langle X^{n}Y-(Y^{m}-X^{e}Z)^{d}-T^{r}-X\, Q(X,Y^{m}-X^{e}Z,T)\rangle$.
Then $\mathcal{D}(B)=\mathbf{k}[x,s,t]$ where $s=y^{m}-x^{e}z$.
In particular, $B$ is not algebraically isomorphic to $\mathbb{A}_{\mathbf{k}}^{3}$.\end{thm}
\begin{proof}
Given a non-zero $\partial\in\mathrm{LND}(B)$, by Proposition \ref{Prop:KeyPro-1-2-3}
(3) and (4), it respects the $\omega_{B}$-filtration determined in
Proposition \ref{Prop: the filtration new example}. Therefore, it
induces a non-zero locally nilpotent derivation $\widehat{\partial}:=\mathrm{gr}_{\omega_{B}}(\partial)$
of $\mathrm{Gr}(B)$. Suppose that $f\in\ker\partial$, then $\widehat{f}:=\mathrm{gr}(f)\in\ker\widehat{\partial}$
is an homogenous element of $\mathrm{Gr}(B)$. 

Assume that $\widehat{f}\notin\mathbf{k}[\widehat{x},\widehat{s},\widehat{t}]$,
then, by Corollary \ref{Cor: the grading new example}, $\widehat{y}$
or $\widehat{z}$ must divides $\widehat{f}$. This yields a contradiction
as follows.

If $\widehat{z}$ divides $\widehat{f}$, then $\widehat{\partial}(\widehat{z})=0$
as $\ker\widehat{\partial}$ is factorially closed. On the other hand,
if $\widehat{y}$ divides $\widehat{f}$, then $\widehat{y}\in\ker\widehat{\partial}$.
Thus, the relation $\widehat{y}^{m}-\widehat{x}^{e}$$\widehat{z}$
implies that $\widehat{z}\in\ker\widehat{\partial}$ as $\ker\widehat{\partial}$
is factorially closed. Therefore, either way the assumption $\hat{f}\notin\mathbf{k}[\widehat{x},\widehat{s},\widehat{t}]$
implies that $\widehat{\partial}(\widehat{z})=0$.

The case where $e=1$ is particular since then $\widehat{\partial}$
extends to a locally nilpotent derivation of the $\mathbf{k}(\widehat{z})$-domain
$\mathbf{k}(\widehat{z})[X,Y,S,T]/\langle X^{n}Y-S^{d}-T^{r},\, Y^{m}-X\widehat{z}\rangle$,
which is isomorphic to 
\[
\mathbf{k}(\widehat{z})[Y,S,T]/\langle\frac{1}{\widehat{z}^{n}}Y^{nm+1}-S^{d}-T^{r}\rangle.
\]
Since the latter is a rigid ring, see \cite[Section 7.1]{Crachiola Maubach},
we get $\widehat{\partial}=0$, a contradiction. 

For the case where $e>1$, let $\varpi\in\mathbb{Z}^{5}$ be another
weight degree function on $\mathrm{Gr}(B)$ defined by: 
\[
\varpi(X)=q,\,\varpi(Y)=-n_{0},\,\varpi(Z)=-mn_{0}-eq,\,\varpi(S)=r,\,\varpi(T)=d,
\]
where $rd=nq-n_{0}$, $q\in\mathbb{Z}$, and $n_{0}\in\{0,\ldots,n-1\}$.
Then, $\mathrm{Gr}_{\varpi}(\mathrm{Gr}(B))=\mathrm{Gr}(B)$, that
is, $\varpi$ is a $\mathbb{Z}$-grading of $\mathrm{Gr}(B)$. Hence,
$\widehat{\partial}$ induces $\widetilde{\partial}:=\mathrm{gr}_{\varpi}(\widehat{\partial})$
a non-zero locally nilpotent derivation of $\mathrm{Gr}(B)$, such
that $\widetilde{\partial}(\widehat{z})=0$. 

\noindent Choose $H\in\ker\widetilde{\partial}$ which is a non-constant
homogeneous, relative to both grading of $B$, and algebraically independent
of $\widehat{z}$, which is possible since $\ker\widetilde{\partial}$
is generated by homogeneous elements and it is of transcendence degree
$2$ over $\mathbf{k}$, see \cite{Makar-Limanov: A new ring invariant}.
Then, the only possibility for $H$ is $H=h(\widehat{s},\widehat{t}\,)$
otherwise we get $\widetilde{\partial}=0$. 

\noindent Since $\gcd(d,r)=1$, there exists a standard homogeneous
polynomial $h_{0}\in\mathbf{k}^{[2]}$ such that $h(\widehat{s},\widehat{t}\,)=h_{0}(\widehat{s}^{d},\widehat{t}^{r})$,
see \cite[Lemma 4.6]{Freudenburg}. Thus we have $h_{0}(\widehat{s}^{d},\widehat{t}^{r})\in\ker\widetilde{\partial}$,
which implies that either $\widetilde{\partial}(\widehat{s})=0$ or
$\widetilde{\partial}(\widehat{t})=0$ (or both), see \cite[Prop. 9.4]{Freudenburg}.
But if $\widetilde{\partial}(\widehat{t})=0$, then $\widetilde{\partial}$
extends to a locally nilpotent derivation of 
\[
A:=\mathbf{k}(\widehat{z},\widehat{t})[X,Y,S]/\langle X^{n}Y-S^{d}-\widehat{t}^{r},\, Y^{m}-X^{e}\widehat{z}\rangle.
\]
It follows from the Jacobian criterion that $\mathrm{Spec}(A)$ has
a non-empty set of singular points as $e>1$. Since $A$ is an integral
domain of transcendence degree one over $\mathbf{k}(\widehat{z},\widehat{t}\,)$,
\cite[Corollary 1.29]{Freudenburg} implies that $A$ is rigid, and
therefore $\widetilde{\partial}=0$, a contradiction. In the same
way we get a contradiction if $\widetilde{\partial}(\widehat{s})=0$.

Therefore, the only possibility is that $\widehat{f}\in\mathbf{k}[\widehat{x},\widehat{s},\widehat{t}]$,
which yields $f\in\mathbf{k}[x,s,t]$. This proves that $\mathcal{D}(B)\subseteq\mathbf{k}[x,s,t]\simeq\mathbf{k}^{[3]}$.
To complete the proof, define $D_{1},D_{2}\in\mathrm{LND}(B)$ by:
\[
D_{1}(x)=D_{1}(t)=0,D_{1}(s)=x^{n+e},D_{1}(y)=x^{e}(ds^{d-1}+x\frac{\partial Q}{\partial s}),\, D_{1}(z)=my^{m-1}(ds^{d-1}+x\frac{\partial Q}{\partial s})-x^{n}
\]
 and 
\[
D_{2}(x)=D_{2}(s)=0,D_{2}(t)=x^{n+e},D_{2}(y)=x^{e}(rt^{r-1}+x\frac{\partial Q}{\partial t}),\, D_{2}(z)=my^{m-1}(rt^{r-1}+x\frac{\partial Q}{\partial t}).
\]
Then obviously $\mathbf{k}[x,s,t]\subseteq\mathcal{D}(B)$.
\end{proof}
What we did establish in the proof of Theorem \ref{Thm:Derkson-invariant-of B}
is actually more than the assertion announced in the Theorem itself.
Indeed,
\begin{lem}
\label{Lem:deg-of-LND} Let $\omega_{B}$ be the degree function on
$B$ defined as in Theorem \ref{Thm:Derkson-invariant-of B}, then:
\[
\deg_{\omega_{B}}\partial\leq-n-e\,;\,\textrm{for every non-zero }\partial\in\mathrm{LND}(B).
\]
\end{lem}
\begin{proof}
Let $\partial\in\mathrm{LND}(B)$ be non-zero. Continuing the notation
of the proof of Theorem \ref{Thm:Derkson-invariant-of B}, $\partial$
induces $\widehat{\partial}:=\mathrm{gr}(\partial)$ a non-zero locally
nilpotent derivation of $\mathrm{Gr}(B)$. We have  established that
$\widehat{\partial}(\widehat{z})\neq0$. Denote by $\tau$ the degree
of $\partial$ with respect to $\omega_{B}$, $\tau:=\deg_{\omega_{B}}\partial$.

Assume for contradiction that $\tau=\deg_{\omega_{B}}\partial>-(n+e)$.
Then, for every $b\in B$ such that $\omega_{B}(b)=i$, we have by
definition of $\widehat{\partial}$ that $\widehat{\partial}(\hat{b})=\begin{cases}
0 & ;\,\,\textrm{if}\,\,\omega_{B}(\partial(b))<i+\tau\\
\widehat{\partial(b)} & ;\,\,\textrm{if}\,\,\omega_{B}(\partial(b))=i+\tau
\end{cases}$. Thus we conclude that either $\widehat{\partial}(\widehat{z})=0$,
which is excluded, or $\widehat{\partial}(\widehat{z})=\widehat{\partial(z)}$.
But since $\omega_{B}(z)=nm+e$, we see that $\partial(z)\in\mathcal{F}_{nm+e+\tau}$,
and $\widehat{\partial(z)}\in B_{[nm+e+\tau]}$. So $\widehat{z}$
divides $\widehat{\partial}(\widehat{z})$ by Corollary \ref{Cor: the grading new example},
which implies that $\widehat{\partial}(\widehat{z})=0$ by reasons
of degree, see \cite[Corollary 1.20]{Freudenburg}, which is absurd.
Therefore, the only possibility is that $\tau=\deg_{\omega_{B}}\partial\leq-n-e$.
And we are done.
\end{proof}
Consider the following chain of inclusions: 
\[
\mathcal{D}(B)=\mathbf{k}[x,s,t]\hookrightarrow R=\mathbf{k}[x,s,t,y]\hookrightarrow B=\mathbf{k}[x,s,t,y,z],
\]
where $R$ is the Russell $\mathbf{k}$-domain corresponding to the
pair $(n,S^{d}+T^{r}+X\, Q(X,S,T))$, we have the following.
\begin{cor}
\label{Cor:LND-restricts-to-} Every $\partial\in\mathrm{LND}(B)$
restricts to a locally nilpotent derivation of $\mathbf{k}[x,y,s,t]=R$
$($resp. $\mathbf{k}[x,s,t]=\mathcal{D}(B)$$)$. Furthermore, 
\[
\partial(R)\subseteq\langle x^{e}\rangle_{R}\,\,\,\,\,\,\textrm{and}\,\,\,\,\,\,\partial(\mathcal{D}(B))\subseteq\langle x^{n+e}\rangle_{\mathcal{D}(B)}\,,
\]
where $\langle x^{e}\rangle_{R}$ $($resp. $\langle x^{n+e}\rangle_{\mathcal{D}(B)}$$)$
is the principle ideal of $R$ $($resp.$\mathcal{D}(B)$$)$ generated
by $x^{e}$ $($resp. $x^{n+e}$$)$. \end{cor}
\begin{proof}
Let $\partial\in\mathrm{LND}(B)$ be non-zero. By Lemma \ref{Lem:deg-of-LND},
we have $\tau=\deg_{\omega_{B}}\partial\leq-n-e$. This means $\partial(\mathcal{F}_{i})\subseteq\mathcal{F}_{i+\tau}\subseteq\mathcal{F}_{i}$,
and hence, $\partial(\mathcal{D}(B))\subseteq\mathcal{D}(B)$ and
$\partial(R)\subseteq R$. Furthermore, 
\[
\partial(\mathbf{k}[x,s,t])\subseteq\mathcal{F}_{-n-e}=\mathbf{k}[x,s,t]x^{n+e}+\mathcal{F}_{-n-e-1}=\langle x^{n+e}\rangle_{\mathbf{k}[x,s,t]}.
\]
Finally, 
\[
\partial(y)\in\mathbf{k}[x,s,t]x^{e}+\mathcal{F}_{-e-1}=\langle x^{e}\rangle_{\mathbf{k}[x,s,t]}.
\]
The latter implies that $\partial(R)\subseteq\langle x^{e}\rangle_{R}$,
as desired.
\end{proof}

\subsection{The Makar-Limanov invariant and $\mathrm{LND}$}

\indent\newline\noindent  Recall that the \emph{Makar-Limanov invariant}
$\mathrm{ML}(A)$ of a $\mathbf{k}$-domain $A$ is defined to be
the intersection of the kernels of all locally nilpotent derivations
of $A$. That is, $\mathrm{ML}(A):=\cap_{\partial\in\mathrm{\mathrm{LND}}(A)}\ker\partial$. 

\noindent  The observation that every locally nilpotent derivation
of $B$ must restrict to a locally nilpotent derivation of the sub-algebra
$R$, introduce a consecutive way to compute the Makar-Limanov invariant.
That is, consider the inclusion $R\hookrightarrow B$. It is well-known
that $\mathrm{ML}(R)=\mathbf{k}[x]$; $n\geq2$, see \cite{Kaliman Makar-Limanov Threefolds,Kaliman Makar-Limanovr: AK invariant,M-L2005}.
On the other hand, by Theorem \ref{Cor:LND-restricts-to-}, every
$\partial\in\mathrm{LND}(B)$ restricts to $\partial|_{R}$ a locally
nilpotent derivation of $R$. Therefore, since $\mathrm{ML}(R)=\cap_{D\in\mathrm{LND}(R)}\ker D\subseteq\cap_{\partial\in\mathrm{LND}(B)}\ker\partial|_{R}$,
we immediately obtain $\mathrm{ML}(R)=\mathbf{k}[x]\subseteq\mathrm{ML}(B)$.
Finally, since $\mathrm{ML}(B)\subseteq\ker D_{1}\cap\ker D_{2}=\mathbf{k}[x]$
where $D_{1},D_{2}\in\mathrm{LND}(B)$ define as in the proof of Theorem
\ref{Thm:Derkson-invariant-of B}, we get Corollary \ref{Cor:ML(B)=00003Dk[x]},
in the cases $n\geq2$, for free.

Nevertheless, for the general case, we present an alternative approach
to compute the Makar-Limanov invariant for the class of examples \ref{Sec:The-class-B_(n,e,Q)}.
That is, it can be deduced from Corollary \ref{Cor:LND-restricts-to-}
as follows.
\begin{cor}
\label{Cor:ML(B)=00003Dk[x]} $\mathrm{ML}(B)=\mathbf{k}[x]$.\end{cor}
\begin{proof}
Let $\partial\in\mathrm{LND}(B)$, then Theorem \ref{Cor:LND-restricts-to-}
in particular, asserts that $\partial(\mathbf{k}[x,s,t])\subseteq\langle x^{n+e}\rangle_{\mathbf{k}[x,s,t]}$.
This implies that $\partial(x)$ is divisible by $x$, thus by reasons
of degree, see also \cite[ Corollary 1.20 ]{Freudenburg}, we conclude
that $\partial(x)=0$ and $\mathbf{k}[x]\subseteq\mathrm{ML}(B)$.
Finally, noting that $\mathrm{ML}(B)\subseteq\mathbf{k}[x]=\ker D_{1}\cap\ker D_{2}$
where $D_{1},D_{2}\in\mathrm{LND}(B)$ define as in the proof of Theorem
\ref{Thm:Derkson-invariant-of B}, we have $\mathrm{ML}(B)=\mathbf{k}[x]$,
as desired.
\end{proof}
The following Corollary, which is a consequence of Corollary \ref{Cor:LND-restricts-to-},
describes the set $\mathrm{LND}(B)$. Denote by $\mathrm{LND}_{\mathbf{k}[x]}(\mathbf{k}[x,s,t])$
the set of locally nilpotent derivations of $\mathbf{k}[x,s,t]\simeq\mathbf{k}^{[3]}$
that have $x$ in their kernels, then:
\begin{cor}
\label{Cor:Description of LND} $\mathrm{LND}(B)=x^{e}\left(\mathrm{LND}(R)\right)=x^{n+e}\left(\mathrm{LND}_{\mathbf{k}[x]}(\mathbf{k}[x,s,t])\right)$.\end{cor}
\begin{proof}
Let $\delta$ be a locally nilpotent derivation of $R$ (resp. $\mathbf{k}[x,s,t]$
that annihilates $x$), then the derivation $x^{e}\delta$ (resp.
$x^{n+e}\delta$) extends to a locally nilpotent derivation of $B$
by taking 
\[
(x^{e}\delta)(z)=\frac{(x^{e}\delta)(y^{m})-(x^{e}\delta)(s)}{x^{e}}=\delta(y^{m})-\delta(s)
\]
 resp. 
\[
(x^{n+e}\delta)(y)=\frac{(x^{n+e}\delta)(s^{d}+t^{r}+xQ)}{x^{n}}=x^{e}\delta(s^{d}+t^{r}+xQ)\,\,,\,\,\textrm{and}
\]
 
\[
(x^{n+e}\delta)(z)=\frac{(x^{n+e}\delta)(y^{m}-s)}{x^{e}}=my^{m-1}\delta(s^{d}+t^{r}+xQ)-x^{n}\delta(s).
\]
We denote $\widetilde{\delta}=x^{e}\delta$ (resp. $\widetilde{\delta}=x^{n+e}\delta$).
Conversely, Corollary \ref{Cor:LND-restricts-to-} ensures that every
$\partial\in\mathrm{LND}(B)$ restricts to $\partial|_{R}\in\mathrm{LND}(B)$
as well as $\partial|_{\mathbf{k}[x,s,t]}\in\mathrm{LND}_{\mathbf{k}[x]}(\mathbf{k}[x,s,t])]$,
such that $\partial|_{R}=x^{e}\delta_{1}$ and $\partial|_{\mathbf{k}[x,s,t]}=x^{n+e}\delta_{2}$
for some $\delta_{1}\in\mathrm{LND}(R)$ and $\delta_{2}\in\mathrm{LND}_{\mathbf{k}[x]}(\mathbf{k}[x,s,t])$,
hence $\delta_{1}|_{\mathbf{k}[x,s,t]}=x^{n}\delta_{2}$. This establishes
the correspondence. Finally, it is straightforward to check that the
latter is a one-to-one correspondence, that is, $\widetilde{\partial|\,}_{R}=\partial$
and $\widetilde{\delta}\,|_{R}=\delta$ (resp. $\widetilde{\partial|\,}_{\mathbf{k}[x,s,t]}=\partial$
and $\widetilde{\delta}\,|_{\mathbf{k}[x,s,t]}=\delta$). And we are
done. 
\end{proof}
The next Corollary describes the kernels of locally nilpotent derivations
of $B$. The proof of \cite[Corollary 9.8]{Freudenburg} also applies
here.
\begin{cor}
\label{Cor:Description of kernels} Let $\partial\in\mathrm{LND}(B)$
be non-zero, then there exists $F\in\mathbf{k}[x,s,t]\subset B$ such
that $F$ is a $\mathbf{k}(x,x^{-1})$-variable of $\mathbf{k}(x,x^{-1})[s,t]$,
and $\ker\partial=\mathbf{k}[x,F]=\mathbf{k}^{2}$.
\end{cor}

\section{\textbf{\label{Sec:Exponential-Modifications} Exponential Modifications}}

\subsection{Definitions and basic properties}

\indent\newline\noindent  Let $A$ be a finitely generated domain
over a field $\mathbf{k}$ of characteristic zero, $I$ be an ideal
in $A$ and $f$ be a non-zero element of $I$.
\begin{defn}
By the \emph{affine modification} of $A$ along $f$ with center $I$,
see \cite{Kaliman,Kaliman Zaidenberg: Affine Modification,Zaidenberg},
we mean the sub-algebra $A':=A[I/f]$ of $A_{f}$ (the localization
of $A$ with respect to $f$) generated by $A$ and the sub-set $I/f$.
Similarly, for any $\mathbf{k}$-domain $A$, the sub-algebra $A':=A[I/f]$
of $A_{f}$ is called the \emph{modification} of $A$ along $f$ with
center $I$. The pair $(f,I)$ is called the \emph{locus} of the modification
and $A$ is called the \emph{base} of the modification. 
\end{defn}
\noindent  If the ideal $I$ is finitely generated, say $I=\langle f,b_{1},\ldots,b_{r}\rangle_{A}$,
then $A'$ is the sub-algebra of $A_{f}\subset\mathrm{Frac}(A)$ which
is generated by $A$ and the elements $b_{1}/f,\ldots,b_{r}/f$. That
is, 
\[
A[I/f]=\{P(b_{1}/f,\ldots,b_{r}/f);\, P(X_{1},\ldots,X_{r})\in A[X_{1},\ldots,X_{r}]\}.
\]
Therefore, we get 
\[
A[I/f]=\{a/f^{d}\in A_{f};\, a\in I^{d}\,\textrm{and }d\in\mathbb{N}\}.
\]
The extension of the ideal $I$ in $A'=A[I/f]$ coincides with the
principal ideal generated by $f$, that is, $I.A[I/f]=\langle f\rangle_{A[I/f]}$.

The next lemma manifests the universal property of modifications,
see \cite[Proposition 2.1 and Corollary 2.2]{Kaliman Zaidenberg: Affine Modification}.
\begin{lem}
\label{Lem:extension-of-iso-to-algebraic-extension-of-degree-one}
Let $\Psi:A\longrightarrow B$ be an isomorphism between domains $A$
and $B$, $I$ be an ideal in $A$, and $f\in I$. Then $\Psi$ extends
in a unique way to an isomorphism $\widetilde{\Psi}:A[I/f]\longrightarrow B[\Psi(I)/\Psi(f)]$.\end{lem}
\begin{proof}
Define $\widetilde{\Psi}:A[I/f]\longrightarrow B[\Psi(I)/\Psi(f)]$
by $\widetilde{\Psi}(a)=\Psi(a)$ for every $a\in A$ and $\widetilde{\Psi}(P(b_{1}/f,\ldots,b_{s}/f))=P_{\Psi}(\Psi(b_{1})/\Psi(f),\ldots,\Psi(b_{s})/\Psi(f))$,
where $P(X_{1},\ldots,X_{s})=\sum_{\textrm{finite}}a_{i}X_{1}^{n_{(1,i)}}\ldots X_{s}^{n_{(s,i)}}\in A[X_{1},\ldots,X_{s}]$
and $P_{\Psi}(X_{1},\ldots,X_{s})=\sum_{\textrm{finite}}\Psi(a_{i})X_{1}^{n_{(1,i)}}\ldots X_{s}^{n_{(s,i)}}\in B[X_{1},\ldots,X_{s}]$.
Then $\widetilde{\Psi}$ is an isomorphism with inverse $\widetilde{\Psi}^{-1}:B[\Psi(I)/\Psi(f)]\longrightarrow A[I/f]$
defined by $\widetilde{\Psi}^{-1}(\Psi(a))=\Psi^{-1}(\Psi(a))=a$
for every $\Psi(a)\in B$ (i.e., $\widetilde{\Psi}^{-1}|_{B}=\Psi^{-1}$)
and $\widetilde{\Psi}^{-1}(H(\Psi(b_{1})/\Psi(f),\ldots,\Psi(b_{s})/\Psi(f))=H_{\Psi^{-1}}(b_{1}/f,\ldots,b_{s}/f)$.
Finally, let $\Phi$ be an isomorphism between $A[I/f]$ and $B[\Psi(I)/\Psi(f)]$,
such that $\Phi|_{A}=\Psi$, then $\Phi(P(b_{1}/f,\ldots,b_{s}/f))=P_{\Phi|_{A}}(\Phi(b_{1}/f),\ldots,\Phi(b_{s}/f))$.
Since $\Phi(b_{i})=\Phi(fb_{i}/f)=\Phi(f)\Phi(b_{i}/f)$, we conclude
that $\Phi(b_{i}/f)=\Phi(b_{i})/\Phi(f)$ for every $i$. Hence $\Phi(P(b_{1}/f,\ldots,b_{s}/f))=P_{\Phi|_{A}}(\Phi(b_{1})/\Phi(f),\ldots,\Phi(b_{s})/\Phi(f))=P_{\Psi}(\Psi(b_{1})/\Psi(f),\ldots,\Psi(b_{s})/\Psi(f))$
and hence $\Phi=\widetilde{\Psi}$, as desired.
\end{proof}

\subsection{Exponential modifications \label{Sub:elementary modification}}

\indent\newline\noindent  We are interested in modifications of $A$
along elements of the form $f=a^{n}$; $n\in\mathbb{N}\backslash\{0\}$
for some element $a$ in $A$. 
\begin{defn}
Let $A$ be an integral domain, $I$ be an ideal in $A$, and $a$
be an irreducible element in $A$ such that $a^{n}\in I$. The modification
$A[I/a^{n}]$ of $A$ along $a^{n}$ with center $I$ will be called
the \emph{exponential modification} of $A$ with respect to $a$.
The \emph{contraction} of $\langle a^{N}\rangle_{A[I/a^{n}]}$ with
respect to the inclusion $A\overset{\iota}{\hookrightarrow}A[I/a^{n}]$
(also called the contraction of $\langle a^{N}\rangle_{A[I/a^{n}]}$
in $A$) is $\langle a^{N}\rangle_{A[I/a^{n}]}^{\mathbf{c}}:=\{b\in A;\,\,\iota(b)\in\langle a^{N}\rangle_{A[I/a^{n}]}\}$. 
\end{defn}
\noindent  The principal ideal $\langle a^{N}\rangle_{A[I/a^{n}]}$
will be denoted simply by $\langle a^{N}\rangle$ (not to be confused
with $\langle a^{N}\rangle_{A}$ the principle ideal in $A$ generated
by $a^{N}$, i.e., $\langle a^{N}\rangle_{A[I/a^{n}]}\neq\langle a^{N}\rangle_{A}$
in general). Note that the contraction of $\langle a^{N}\rangle$
in $A$ coincides with 
\[
\langle a^{N}\rangle^{\textrm{\ensuremath{\mathbf{c}}}}=A\cap\langle a^{N}\rangle.
\]
Also, the extension of $I$ to $A[I/a^{n}]$ (i.e., the ideal in $A[I/a^{n}]$
generated by $I$) coincides with the principle ideal generated by
$a^{n}$, that is, $I.A[I/a^{n}]=\langle a^{N}\rangle$. 

Consider the following chain of principal ideals in $A[I/a^{n}]$:
\[
A[I/a^{n}]=\langle1\rangle\supset\langle a\rangle\supset\langle a^{2}\rangle\supset\cdots\supset\langle a^{n}\rangle,
\]
it induces the following chain of ideals in $A$.
\[
A=\langle1\rangle^{\mathbf{c}}\supset\langle a\rangle^{\mathbf{c}}\supset\langle a^{2}\rangle^{\mathbf{c}}\supset\cdots\supset\langle a^{n}\rangle^{\mathbf{c}}.
\]
Note that $\langle a^{n}\rangle^{\mathbf{c}}=I$. To the latter chain
of ideals we associate the following chain of sub-algebras of $A[I/a^{n}]\subset A[a^{-1}]$:
\[
A\subset A[\langle a\rangle^{\mathbf{c}}/a]\subset A[\langle a^{2}\rangle^{\mathbf{c}}/a^{2}]\subset\cdots\subset A[\langle a^{n}\rangle^{\mathbf{c}}/a^{n}]=A[I/a^{n}].
\]
Note that $A[\langle a^{N}\rangle^{\mathbf{c}}/a^{N}]$ is the exponential
modification of $A$, along $a^{N}$ with center $\langle a^{N}\rangle^{\mathbf{c}}$
for every $N\in\mathbb{N}$.
\begin{defn}
The chain $A=\langle1\rangle^{\mathbf{c}}\supset\langle a\rangle^{\mathbf{c}}\supset\langle a^{2}\rangle^{\mathbf{c}}\supset\cdots\supset\langle a^{n}\rangle^{\mathbf{c}}=I$
of ideals in $A$ will be called the \emph{contraction chain} associated
to $A[I/a^{n}]$ and the chain $A\subset A[\langle a\rangle^{\mathbf{c}}/a]\subset A[\langle a^{2}\rangle^{\mathbf{c}}/a^{2}]\subset\cdots\subset A[\langle a^{n}\rangle^{\mathbf{c}}/a^{n}]=A[I/a^{n}]$
of sub-algebras of $A[I/a^{n}]$ will be called the \emph{exponential
chain} of $A[I/a^{n}]$.
\end{defn}
The next theorem shows in particular that isomorphisms between exponential
modifications, which preserve bases of modifications together with
their principal ideals generated by their centers, respect the associated
contraction and exponential chains.
\begin{thm}
\label{Thm: iso-mod preservs bases and centers preserves contraction chains}
Let $\Psi:A[I/a^{n}]\longrightarrow B'$ be an isomorphism between
the exponential modifications $A[I/a^{n}]$ and $B'$. Assume that
$\Psi(A)=B\subset B'$ and $\Psi(a)=b$. Then\emph{:} 

$(1)$ $\Psi$ respects the contraction chains, that is, $\Psi(\langle a^{N}\rangle_{A[I/a^{n}]}^{\mathbf{c}})=\langle b^{N}\rangle_{B'}^{\mathbf{c}}$
for every $N\in\mathbb{N}$.

$(2)$ $\Psi$ respects the exponential chains, that is, $\Psi(A[\langle a^{N}\rangle_{A[I/a^{n}]}^{\mathbf{c}}/a^{N}])=B[\langle b^{N}\rangle_{B'}^{\mathbf{c}}/b^{N}]$
for every $N\in\mathbb{N}$.

In particular, $B'$ can be realized as the exponential modification
of $B$ with locus $(b^{n},\langle b^{n}\rangle_{B'}^{\mathbf{c}})$,
that is, $B'=B[\langle b^{n}\rangle_{B'}^{\mathbf{c}}/b^{n}]$.\end{thm}
\begin{proof}
Assertion (1), since $\Psi(a)=b$, $\Psi(A)=B$ and $\langle a^{N}\rangle_{A[I/a^{n}]}^{\mathbf{c}}=A\cap\langle a^{N}\rangle_{A[I/a^{n}]}$,
we have $\Psi(\langle a^{N}\rangle_{A[I/a^{n}]}^{\mathbf{c}})=B\cap\langle b^{N}\rangle_{B'}$
for every $N\in\mathbb{N}$. Since $B\cap\langle b^{N}\rangle_{B'}=\langle b^{N}\rangle_{B'}^{\mathbf{c}}$,
we conclude that $\Psi(\langle a^{N}\rangle_{A[I/a^{n}]}^{\mathbf{c}})=\langle b^{N}\rangle_{B'}^{\mathbf{c}}$
for every $N\in\mathbb{N}$. 

\noindent  Assertion (2), since $\Psi(A)=B$ and $\Psi(a)=b$, Lemma
\ref{Lem:extension-of-iso-to-algebraic-extension-of-degree-one} asserts
that the restriction $\Psi|_{A}$ of $\Psi$ to $A$ extends to an
isomorphism $\Phi_{N}:=\widetilde{\Psi|_{A}}$ between $A[\langle a^{N}\rangle_{A[I/a^{n}]}^{\mathbf{c}}/a^{N}]$
and $B[\Psi(\langle a^{N}\rangle_{A[I/a^{n}]}^{\mathbf{c}})/\Psi(a)^{N}]$
for every $N\in\mathbb{N}$. By assertion (1), the latter ring coincides
with $B[\langle b^{N}\rangle_{B'}^{\mathbf{c}}/b^{N}]$. Moreover,
noting that this extension is unique, $\Phi_{N}$ coincides with the
restriction $\Psi|_{A[\langle a^{N}\rangle_{A[I/a^{n}]}^{\mathbf{c}}/a^{N}]}$
of $\Psi$ to $A[\langle a^{N}\rangle_{A[I/a^{n}]}^{\mathbf{c}}/a^{N}]$,
i.e., $\Phi_{N}=\Psi|_{A[\langle a^{N}\rangle_{A[I/a^{n}]}^{\mathbf{c}}/a^{N}]}$.
Hence, $\Psi(A[\langle a^{N}\rangle_{A[I/a^{n}]}^{\mathbf{c}}/a^{N}])=B[\langle b^{N}\rangle_{B'}^{\mathbf{c}}/b^{N}]$,
for every $N\in\mathbb{N}$, and in particular \emph{$B'=B[\langle b^{n}\rangle_{B'}^{\mathbf{c}}/b^{n}]$},
as desired.
\end{proof}
The previous theorem asserts that if $A[I/a^{n}]\simeq B'$, then
there exist an element $b\in B'$, a sub-algebra $B\subset B'$, and
an ideal $J\subset B$ contains $b^{n}$, such that $A\simeq B$ and
$B'$ can be realized as the modification of $B$ with locus $(b^{n},J:=\langle b^{n}\rangle_{B'}^{\mathbf{c}})$.
Furthermore, every $\mathbf{k}$-isomorphism $\Psi:A[I/a^{n}]\longrightarrow B'$
such that $\Psi(A)=B$, restricts to a $\mathbf{k}$-isomorphism between
$A[\langle a^{N}\rangle_{A[I/a^{n}]}^{\mathbf{c}}/a^{N}]$ and $B[\langle b^{N}\rangle_{B'}^{\mathbf{c}}/b^{N}]$
for every $N$, where $\Psi(a)=b$ and $\Psi(I)=J$. Therefore, we
have the following commutative diagram:
\[
\begin{array}{ccc}
A[I/a^{n}] & \overset{\Psi}{\longrightarrow} & B'=B[\langle b^{n}\rangle_{B'}^{\mathbf{c}}/b^{n}]\\
\cup & \circlearrowright & \cup\\
\vdots & \vdots & \vdots\\
\cup & \circlearrowright & \cup\\
A[\langle a^{2}\rangle_{A[I/a^{n}]}^{\mathcal{\mathbf{c}}}/a^{2}] & \overset{\sim}{\longrightarrow} & B[\langle b^{2}\rangle_{B'}^{\textrm{\ensuremath{\mathbf{c}}}}/b^{2}]\\
\cup & \circlearrowright & \cup\\
A[\langle a\rangle_{A[I/a^{n}]}^{\mathcal{\mathbf{c}}}/a] & \overset{\sim}{\longrightarrow} & B[\langle b\rangle_{B'}^{\textrm{\ensuremath{\mathbf{c}}}}/b]\\
\cup & \circlearrowright & \cup\\
A & \underset{\Psi|_{_{A}}}{\overset{\sim}{\longrightarrow}} & B
\end{array}
\]

\section{\textbf{Isomorphism classes and Automorphism groups }}

Let $m,d,r\geq2$ be fixed such that $\gcd(d,r)=1$. For every $e\geq0$,
$n\geq1$ such that $(n,e)\neq(1,0)$, and every $Q\in\mathbf{k}[X,S,T]$,
we denote by $B_{(n,e,Q)}$ the following $\mathbf{k}$-domain: 
\[
B_{(n,e,Q)}:=\mathbf{k}[x,y,z,t]\simeq\mathbf{k}[X,Y,Z,T]/\langle X^{n}Y-(Y^{m}-X^{e}Z)^{d}-T^{r}-X\, Q(X,Y^{m}-X^{e}Z,T)\rangle
\]
which is isomorphic to 
\[
\mathbf{k}[X,Y,Z,S,T]/\langle X^{n}Y-S^{d}-T^{r}-X\, Q(X,S,T),\, Y^{m}-X^{e}Z-S\rangle.
\]
Also, we denote by $\mathbf{R}_{(n,S^{d}+T^{r}+XQ)}$ the Russel $\mathbf{k}$-domain:
\[
\mathbf{R}_{(n,S^{d}+T^{r}+XQ)}:=\mathbf{k}[x,y,s,t]\simeq\mathbf{k}[X,Y,S,T]/\langle X^{n}Y-S^{d}-T^{r}-X\, Q(X,S,T)\rangle
\]
Consider the following two chains of inclusions, for $i\in\{1,2\}$:
\[
\mathbf{k}[x,s,t]\hookrightarrow\mathbf{R}_{(n_{i},S^{d}+T^{r}+XQ_{i})}=\mathbf{k}[x,s,t,y_{i}]\hookrightarrow B_{(n_{i},e_{i},Q_{i})}=\mathbf{k}[x,s,t,y_{i},z_{i}]\hookrightarrow B_{(n_{i},e_{i},Q_{i})}[x^{-1}]=\mathbf{k}[x,x^{-1},s,t].
\]
 The last inclusion is realized by the localization homomorphism with
respect to $x$, where 
\[
y_{i}=x^{-n_{i}}(s^{d}+t^{r}+xQ_{i}),\, z_{i}=x^{-n_{i}m-e_{i}}((s^{d}+t^{r}+xQ_{i})^{m}-x^{n_{i}m}s)\in\mathbf{k}[x,x^{-1},s,t]\,\,\,;\,\,\, i\in\{1,2\}.
\]

\noindent Theorem \ref{Thm:Derkson-invariant-of B} and Corollary
\ref{Cor:ML(B)=00003Dk[x]} implies that $\mathcal{D}(B_{(n_{i},e_{i},Q_{i})})=\mathcal{D}(\mathbf{R}_{(n_{i},S^{d}+T^{r}+XQ_{i})})=\mathbf{k}[x,s,t]\simeq\mathbf{k}^{[3]}$
and $\mathrm{ML}(B_{(n_{i},e_{i},Q_{i})})=\mathrm{ML}(\mathbf{R}_{(n_{i},S^{d}+T^{r}+XQ_{i})})=\mathbf{k}[x]$.

\subsection{\label{Sub:Basic-facts} Basic facts}

\indent\newline\noindent  Some conditions that two $\mathbf{k}$-domains,
with the same Derksen and Makar-Limanov invariants, must verify to
be isomorphic can be deduced from properties of their locally nilpotent
derivations. Indeed, the next proposition shows how a prior knowledge
of degrees of all locally nilpotent derivations relative to some degree
function can be used to obtain some conditions that two $\mathbf{k}$-domains
must satisfy to be isomorphic.
\begin{prop}
\label{Prop:iso-preseve-<x>-and-n_1+e_1=00003Dn_2+e_2} Let $\Psi:B_{(n_{1},e_{1},Q_{1})}\longrightarrow B_{(n_{2},e_{2},Q_{2})}$
be a $\mathbf{k}$-isomorphism. Then:\end{prop}
\begin{enumerate}
\item \emph{There exists $\lambda\in\mathbf{k}\setminus\{0\}$ such that
$\Psi(x)=\lambda x$.}
\item \emph{$n_{1}+e_{1}=n_{2}+e_{2}$.}\end{enumerate}
\begin{proof}
Since every $\mathbf{k}$-isomorphism $\Psi$ between $B_{(n_{1},e_{1},Q_{1})}$
and $B_{(n_{2},e_{2},Q_{2})}$ must preserve the Makar-Limanov and
the Derksen invariants, we deduce by virtue of Corollary \ref{Cor:ML(B)=00003Dk[x]}
and Theorem \ref{Thm:Derkson-invariant-of B} that $\Psi$ restricts
to a $\mathbf{k}$-automorphism of $\mathbf{k}[x]$ (resp. $\mathbf{k}[x,s,t]\simeq\mathbf{k}^{[3]}$).
This implies that $\Psi(x)=\lambda x+c$ for some $\lambda\in\mathbf{k}\backslash\{0\}$
and $c\in\mathbf{k}$, and that $\Psi(s),\Psi(t)\in\mathbf{k}[x,s,t]$.

\noindent Let $\partial_{1}\in\mathrm{LND}(B_{(n_{1},e_{1},Q_{1})})$
be a non-zero, then $\partial_{2}:=\Psi\partial_{1}\Psi^{-1}$ is
also a non-zero locally nilpotent derivation of $B_{(n_{2},e_{2},Q_{2})}$.
On the other hand, Corollary \ref{Cor:LND-restricts-to-} ensures
that $\partial_{i}$ restricts to $\mathbf{k}[x,s,t]$ in such a way
that $\partial_{i}(\mathbf{k}[x,s,t])\subseteq\langle x^{n_{i}+e_{i}}\rangle_{\mathbf{k}[x,s,t]}=x^{n_{i}+e_{i}}.\mathbf{k}[x,s,t]$
for every $i\in\{1,2\}$.

\noindent Define $\partial_{1}\in\mathrm{LND}(B_{(n_{1},e_{1},Q_{1})})$
by:
\[
\partial_{1}(x)=\partial_{1}(t)=0,\partial_{1}(s)=x^{n_{1}+e_{1}},\partial_{1}(y_{1})=x^{e_{1}}(ds^{d-1}+x\frac{\partial Q_{1}}{\partial s}),\,\partial_{1}(z_{1})=my_{1}^{m-1}(ds^{d-1}+x\frac{\partial Q_{1}}{\partial s})-x^{n_{1}}.
\]
Then $\partial_{2}:=\Psi\partial_{1}\Psi^{-1}\in\mathrm{LND}(B_{(n_{2},e_{2},Q_{2})})$
and we have $\partial_{2}\Psi=\Psi\partial_{1}$. Therefore, we obtain
the relation $\left(\partial_{2}\Psi\right)(s)=\left(\Psi\partial_{1}\right)(s)$,
where the second part is $\Psi\partial_{1}(s)=\Psi(x^{n_{1}+e_{1}})=(\lambda x+c)^{n_{1}+e_{1}}$.
As discussed before $\Psi(s)\in\mathbf{k}[x,s,t]$ and $\partial_{2}(\mathbf{k}[x,s,t])\subseteq x^{n_{2}+e_{2}}.\mathbf{k}[x,s,t]$,
thus the first part of the forgoing relation is $\partial_{2}(\Psi(s))=x^{n_{2}+e_{2}}g(x,s,t)$
for some $g\in\mathbf{k}[x,s,t]$. Therefore, we get $x^{n_{2}+e_{2}}g(x,s,t)=(\lambda x+c)^{n_{1}+e_{1}}$
in $\mathbf{k}[x,s,t]$, which means that $x^{n_{2}+e_{2}}$ divides
$(\lambda x+c)^{n_{1}+e_{1}}$ in $\mathbf{k}[x]$. This is possible
if and only if $c=0$, hence $(1)$ follows, and $n_{2}+e_{2}\leq n_{1}+e_{1}$.
Finally, by symmetry we get $n_{1}+e_{1}=n_{2}+e_{2}$, as desired.
\end{proof}
As a special case of Proposition \ref{Prop:iso-preseve-<x>-and-n_1+e_1=00003Dn_2+e_2},
we have the following.
\begin{cor}
\label{Cor:iso-e=00003D0-then-n_1=00003Dn_2-1} Let $\Psi:\mathbf{R}_{(n_{1},S^{d}+T^{r}+XQ_{1})}\longrightarrow\mathbf{R}_{(n_{2},S^{d}+T^{r}+XQ_{2})}$
be a $\mathbf{k}$-isomorphism. Then:\end{cor}
\begin{enumerate}
\item \emph{There exists $\lambda\in\mathbf{k}\setminus\{0\}$ such that
$\Psi(x)=\lambda x$.}
\item \emph{$n_{1}=n_{2}$.}\end{enumerate}
\begin{rem}
Assertion (1) of Corollary \ref{Cor:iso-e=00003D0-then-n_1=00003Dn_2-1}
is well-known due to P. Russell. Nevertheless, assertion (2) is new.
The proof of Proposition \ref{Prop:iso-preseve-<x>-and-n_1+e_1=00003Dn_2+e_2}
present an alternative proof for assertion (1), using properties of
locally nilpotent derivation, that delivers assertion (2) for free. 
\end{rem}

\subsection{The chain of invariant sub-algebras associated to \textmd{$B_{(n,e,Q)}$}}

\indent\newline\noindent  Let $I$ be the ideal in $\mathbf{k}[x,s,t]$
generated by $x^{nm+e}$, $x^{n(m-1)+e}(s^{d}+t^{r}+xQ)$, and $(s^{d}+t^{r}+xQ)^{m}-x^{nm}s$,
that is, 
\[
I=\left\langle x^{nm+e},\,\, x^{n(m-1)+e}(s^{d}+t^{r}+xQ),\,\,(s^{d}+t^{r}+xQ)^{m}-x^{nm}s\right\rangle _{\mathbf{k}[x,s,t]}.
\]
Then, the affine modification of $\mathcal{D}(B_{(n,e,Q)})=\mathbf{k}[x,s,t]$
along $x^{nm+e}$ with center $I$ is by definition 
\[
\mathbf{k}[x,s,t]\left[I/x^{nm+e}\right]
\]
where $I/x^{nm+e}=x^{-nm-e}.I$ is the sub-set of $\mathbf{k}[x,x^{-1},s,t]$
consists of elements $x^{-nm-e}b$ where $b\in I$. Therefore, 
\[
\mathbf{k}[x,s,t]\left[I/x^{nm+e}\right]=\mathbf{k}[x,s,t][(s^{d}+t^{r}+xQ)/x^{n},((s^{d}+t^{r}+xQ)^{m}-x^{nm}s)/x^{nm+e}]=\mathbf{k}[x,s,t,y,z]=B.
\]
That is,
\begin{prop}
\label{Prop:affine-mod} The $\mathbf{k}$-domain $B_{(n,e,Q)}$ is
the exponential modification of its Derksen invariant $\mathbf{k}[x,s,t]$
along $x^{nm+e}$ with center $I$.
\end{prop}
\noindent The contraction chain associated to $B_{(n,e,Q)}$ is 
\[
\mathbf{k}[x,s,t]=\langle1\rangle^{\mathbf{c}}\supset\langle x\rangle^{\mathbf{c}}\supset\langle x^{2}\rangle^{\mathbf{c}}\supset\cdots\supset\langle x^{nm+e}\rangle^{\mathbf{c}}.
\]
The exponential chain of $B_{(n,e,Q)}\subset\mathbf{k}[x^{-1},x,s,t]$
is 
\[
\mathbf{k}[x,s,t]\subset\mathbf{k}[x,s,t][\langle x\rangle^{\mathbf{c}}/x]\subset\cdots\subset\mathbf{k}[x,s,t][\langle x^{nm+e}\rangle^{\mathbf{c}}/x^{nm+e}]=B_{(n,e,Q)}\subset\mathbf{k}[x^{-1},x,s,t].
\]

Consider again the following two chains of inclusions, for $i\in\{1,2\}$:
\[
\mathbf{k}[x,s,t]\hookrightarrow B_{(n_{i},e_{i},Q_{i})}=\mathbf{k}[x,s,t,y_{i},z_{i}]\hookrightarrow B_{(n_{i},e_{i},Q_{i})}[x^{-1}]=\mathbf{k}[x,x^{-1},s,t].
\]
Denote by 
\[
I_{i}=\left\langle x^{n_{i}m_{i}+e_{i}},\,\, x^{n_{i}(m_{i}-1)+e_{i}}(s^{d}+t^{r}+xQ_{i}),\,\,(s^{d}+t^{r}+xQ_{i})^{m_{i}}-x^{n_{i}m_{i}}s\right\rangle _{\mathbf{k}[x,s,t]},
\]
and $\langle x^{N}\rangle_{B_{(n_{1},e_{1},Q_{1})}}^{\mathcal{\mathbf{c}}}$
(resp. $\langle x^{N}\rangle_{B_{(n_{2},e_{2},Q_{2})}}^{\textrm{\ensuremath{\mathbf{c}}}}$)
the contraction of the ideal $\langle x^{N}\rangle_{B_{(n_{1},e_{1},Q)}}$
(resp. $\langle x^{N}\rangle_{B_{(n_{2},e_{2},Q)}}$) in $\mathbf{k}[x,s,t]$. 

Via the previous description, we have the following.
\begin{thm}
\label{Thm: iso of new examples  preserves both chains} Let $\Psi:B_{(n_{1},e_{1},Q_{1})}\longrightarrow B_{(n_{2},e_{2},Q_{2})}$
be a $\mathbf{k}$-isomorphism, then $\Psi$ respects their contraction
and exponential chains, that is,\end{thm}
\begin{enumerate}
\item \emph{$\Psi\left(\langle x^{N}\rangle_{B_{(n_{1},e_{1},Q_{1})}}^{\mathcal{\mathbf{c}}}\right)=\langle x^{N}\rangle_{B_{(n_{2},e_{2},Q_{2})}}^{\textrm{\ensuremath{\mathbf{c}}}}$
for every $N\in\mathbb{N}$. In particular, $\Psi(I_{1})=I_{2}$.}
\item \emph{$\Psi(\mathbf{k}[x,s,t][\langle x^{N}\rangle_{B_{(n_{1},e_{1},Q_{1})}}^{\mathcal{\mathbf{c}}}/x^{N}])=\mathbf{k}[x,s,t][\langle x^{N}\rangle_{B_{(n_{2},e_{2},Q_{2})}}^{\textrm{\ensuremath{\mathbf{c}}}}/x^{N}]$
for every $N\in\mathbb{N}$. In particular, $B_{(n_{2},e_{2},Q_{2})}=\mathbf{k}[x,s,t][\langle x^{n_{1}m_{1}+e_{1}}\rangle_{B_{(n_{2},e_{2},Q_{2})}}^{\textrm{\ensuremath{\mathbf{c}}}}/x^{n_{1}m_{1}+e_{1}}]$.}\end{enumerate}
\begin{proof}
Theorem \ref{Thm:Derkson-invariant-of B} and Proposition \ref{Prop:iso-preseve-<x>-and-n_1+e_1=00003Dn_2+e_2}
imply that $\Psi$ restricts to a $\mathbf{k}$-automorphism of $\mathbf{k}[x,s,t]$
and that $\Psi(x)=\lambda x$. Therefore, assertion (1) and (2) follow
directly from Theorem \ref{Thm: iso-mod preservs bases and centers preserves contraction chains}.
\end{proof}
Therefore, we have the following commutative diagram:
\[
\begin{array}{ccc}
B_{(n_{1},e_{1},Q_{1})}=\mathbf{k}[x,s,t,y_{1},z_{1}] & \overset{\Psi}{\longrightarrow} & B_{(n_{2},e_{2},Q_{2})}=\mathbf{k}[x,s,t,y_{2},z_{2}]\\
\cup & \circlearrowright & \cup\\
\vdots & \vdots & \vdots\\
\cup & \circlearrowright & \cup\\
\mathbf{k}[x,s,t][\langle x^{2}\rangle_{B_{(n_{1},e_{1},Q_{1})}}^{\mathcal{\mathbf{c}}}/x^{2}] & \overset{\sim}{\longrightarrow} & \mathbf{k}[x,s,t][\langle x^{2}\rangle_{B_{(n_{2},e_{2},Q_{2})}}^{\textrm{\ensuremath{\mathbf{c}}}}/x^{2}]\\
\cup & \circlearrowright & \cup\\
\mathbf{k}[x,s,t][\langle x\rangle_{B_{(n_{1},e_{1},Q_{1})}}^{\mathcal{\mathbf{c}}}/x] & \overset{\sim}{\longrightarrow} & \mathbf{k}[x,s,t][\langle x\rangle_{B_{(n_{2},e_{2},Q_{2})}}^{\textrm{\ensuremath{\mathbf{c}}}}/x]\\
\cup & \circlearrowright & \cup\\
\mathbf{k}[x,s,t] & \overset{\sim}{\longrightarrow} & \mathbf{k}[x,s,t]
\end{array}
\]
In particular, the exponential chain $\mathbf{k}[x,s,t]\subset\mathbf{k}[x,s,t][\langle x\rangle^{\mathbf{c}}/x]\subset\cdots\subset\mathbf{k}[x,s,t][\langle x^{nm+e}\rangle^{\mathbf{c}}/x^{nm+e}]=B_{(n,e,Q)}$
characterizes $B_{(n,e,Q)}$. That is, 
\begin{cor}
\label{Cor:The-exponential-chain invariant by auto}The exponential
chain $\mathbf{k}[x,s,t]\subset\mathbf{k}[x,s,t][\langle x\rangle^{\mathbf{c}}/x]\subset\cdots\subset\mathbf{k}[x,s,t][\langle x^{nm+e}\rangle^{\mathbf{c}}/x^{nm+e}]=B_{(n,e,Q)}$
of $B_{(n,e,Q)}$ is invariant by every $\mathbf{k}$-automorphism
$\Psi$ of $B_{(n,e,Q)}$. That is, every $\Psi\in\mathrm{Aut}_{\mathbf{k}}(B_{(n,e,Q)})$
restricts to a $\mathbf{k}$-automorphism of every member of the exponential
chain.
\end{cor}

\subsection{Computing contraction and exponential chains associated to $B_{(n,e,Q)}$}

\indent\newline\noindent  The next lemma describes the contraction
of the ideal $\langle x^{N}\rangle$ in $\mathbf{k}[x,s,t]$ for every
$N\in\mathbb{N}$.
\begin{lem}
\label{Lem:The-contraction-of-I_N} Denote $F:=s^{d}+t^{r}+xQ$, and
$G:=(s^{d}+t^{r}+xQ)^{m}-x^{nm}s=F^{m}-x^{nm}s$. Then,\end{lem}
\begin{enumerate}
\item \emph{$\langle x^{n_{0}}\rangle^{\textrm{\ensuremath{\mathbf{c}}}}=\langle F,\, x^{n_{0}}\rangle_{\mathbf{k}[x,s,t]}$
for every $n_{0}\in\{1,\ldots,n\}$.}
\item \emph{$\langle x^{nm_{0}+e_{0}}\rangle^{\textrm{\ensuremath{\mathbf{c}}}}=\langle F^{m_{0}+1},x^{n_{0}}F^{m_{0}},\ldots,x^{(m_{0}-1)n+n_{0}}F,x^{nm_{0}+n_{0}}\rangle_{\mathbf{k}[x,s,t]}$
for every $m_{0}\in\{1,\ldots,m-1\}$ and $n_{0}\in\{1,\ldots,n\}$.}
\item \emph{$\langle x^{nm+e_{0}}\rangle^{\textrm{\ensuremath{\mathbf{c}}}}=\langle G,\, x^{e_{0}}F^{m},x^{n+e_{0}}F^{m-1},\ldots,x^{(m-1)n+e_{0}}F,x^{nm+e_{0}}\rangle_{\mathbf{k}[x,s,t]}$
for every $e_{0}\in\{1,\ldots,e\}$.}\end{enumerate}
\begin{proof}
We only prove $(3)$ for the special case where $e_{0}=e$, the rest
can be proved in the same way. The proof is basically a consequence
of the full description of the proper $\mathbb{Z}$-filtration defined
on $B_{(n,e,Q)}$ as in Definition \ref{Def:weight-degree-on-B}. 

\noindent Let $\omega_{B_{(n,e,Q)}}$ be the degree function on $B_{(n,e,Q)}$
defined as in Definition \ref{Def:weight-degree-on-B}, and suppose
that $f\in\langle x^{nm+e}\rangle^{\textrm{\ensuremath{\mathbf{c}}}}$.
Then $f\in\mathbf{k}[x,s,t]\cap\langle x^{nm+e}\rangle$ and there
exists $b\in B_{(n,e,Q)}$ such that $f=x^{nm+e}b$. On the other
hand, since $f\in\mathbf{k}[x,s,t]$, we have $\omega_{B_{(n,e,Q)}}(f)\leq0$.
Noting that $\omega_{B_{(n,e,Q)}}(x^{nm+e})=-nm-e$, we deduce that
$\omega_{B_{(n,e,Q)}}(b)\leq nm+e$. Therefore, by Proposition \ref{Prop: the filtration new example},
$b$ can be expressed as follows. 
\[
b=z\,(\sum_{i=0}^{e-1}x^{i}f_{i}(s,t))+\sum_{j=1}^{m}y^{j}(\sum_{l=0}^{n-1}x^{l}\, g_{(j,l)}(s,t))+h(x,s,t).
\]
Hence, 
\[
x^{nm+e}b=x^{nm+e}\, z\,(\sum_{i=0}^{e-1}x^{i}f_{i}(s,t))+x^{e}\sum_{j=1}^{m}x^{nm-nj}\, x^{nj}y^{j}(\sum_{l=0}^{n-1}x^{l}\, g_{(j,l)}(s,t))+x^{nm+e}\, h(x,s,t).
\]
Thus,
\[
x^{nm+e}b=G\,(\sum_{i=0}^{e-1}x^{i}f_{i}(s,t))+x^{e}\sum_{j=1}^{m}x^{n(m-j)}\, F^{j}(\sum_{l=0}^{n-1}x^{l}\, g_{(j,l)}(s,t))+x^{nm+e}h(x,s,t).
\]
Therefore, we conclude that
\[
x^{nm+e}b\in\langle x^{nm+e},x^{i}G,x^{(m-j)n+e+l}F^{j};\textrm{ \ensuremath{i\in\{0,\ldots,e-1\}}, \ensuremath{l\in\{0,\ldots,n-1\}}, and \ensuremath{j\in\{1,\ldots,m\}}}\rangle_{\mathbf{k}[x,s,t]}
\]
Finally, 
\[
\langle x^{nm+e}\rangle^{\textrm{\ensuremath{\mathbf{c}}}}=\langle x^{nm+e},\, G,\, x^{(m-j)n+e}F^{j}\,;\,\ensuremath{j\in\{1,\ldots,m\}}\rangle_{\mathbf{k}[x,s,t]}.
\]

\end{proof}
The next lemma determines the sub-algebra $\mathbf{k}[x,s,t][\langle x^{N}\rangle^{\mathbf{c}}/x^{N}]$
for every $N\in\mathbb{N}$.
\begin{lem}
\label{Lem:k[x,s,t][I_N/x_N]} The sub-algebra $\mathbf{k}[x,s,t][\langle x^{N}\rangle^{\mathbf{c}}/x^{N}]\subset B_{(n,e,Q)}$
is given by:

$(1)$ $\mathbf{k}[x,s,t][\langle x^{n_{0}}\rangle^{\mathbf{c}}/x^{n_{0}}]=\mathbf{k}[x,s,t,x^{n-n_{0}}y]=\mathbf{R}_{(n_{0},S^{d}+T^{r}+XQ)}$
for every $n_{0}\in\{1,\ldots,n-1\}$.

$(2)$ $\mathbf{k}[x,s,t][\langle x^{n}\rangle^{\mathbf{c}}/x^{n}]=\cdots=\mathbf{k}[x,s,t][\langle x^{nm}\rangle^{\mathbf{c}}/x^{nm}]=\mathbf{k}[x,s,t,y]=\mathbf{R}_{(n,S^{d}+T^{r}+XQ)}$.

$(3)$ $\mathbf{k}[x,s,t][\langle x^{nm+e_{0}}\rangle^{\mathbf{c}}/x^{nm+e_{0}}]=\mathbf{k}[x,s,t][y,x^{e-e_{0}}z]=B_{(n,e_{0},Q)}$
for every $e_{0}\in\{1,\ldots,e\}$.

\noindent Consequently, the exponential chain of $B_{(n,e,Q)}$ is
\[
\mathbf{k}[x,s,t]\subset\mathbf{R}_{(1,S^{d}+T^{r}+XQ)}\subset\cdots\subset\mathbf{R}_{(n,S^{d}+T^{r}+XQ)}\subset B_{(n,1,Q)}\subset\cdots\subset B_{(n,e,Q)}
\]
where $\mathbf{R}_{(n_{0},S^{d}+T^{r}+XQ)}$ is the Russell $\mathbf{k}$-domain
corresponding to the pair $(n_{0},S^{d}+T^{r}+XQ)$. \end{lem}
\begin{proof}
For $(1)$, by Lemma \ref{Lem:The-contraction-of-I_N}, $\langle x^{n_{0}}\rangle^{\mathbf{c}}=\langle F,\, x^{n_{0}}\rangle_{\mathbf{k}[x,s,t]}$
for every $n_{0}\in\{1,\ldots,n\}$. Therefore, 
\[
\mathbf{k}[x,s,t][\langle x^{n_{0}}\rangle^{\mathbf{c}}/x^{n_{0}}]=\mathbf{k}[x,s,t][F/x^{n_{0}}]=\mathbf{k}[x,s,t][x^{n-n_{0}}y]=\mathbf{R}_{(n_{0},S^{d}+T^{r}+XQ)}.
\]

\noindent For $(2)$, it is enough to show that $\mathbf{k}[x,s,t][\langle x^{n}\rangle^{\mathbf{c}}/x^{n}]=\mathbf{k}[x,s,t][\langle x^{nm}\rangle^{\mathbf{c}}/x^{nm}]$.
Lemma \ref{Lem:The-contraction-of-I_N}, asserts that $\langle x^{nm}\rangle^{\mathbf{c}}=\langle F^{m},x^{n}F^{m-1},\ldots,x^{(m-1)n}F,x^{nm}\rangle_{\mathbf{k}[x,s,t]}$.
Therefore, 
\[
\mathbf{k}[x,s,t][\langle x^{nm}\rangle^{\mathbf{c}}/x^{nm}]=\mathbf{k}[x,s,t][F^{m}/x^{nm},x^{n}F^{m-1}/x^{nm},\ldots,x^{(m-1)n}F/x^{nm}].
\]
Thus, we get 
\[
\mathbf{k}[x,s,t][\langle x^{nm}\rangle^{\mathbf{c}}/x^{nm}]=\mathbf{k}[x,s,t][y^{m},\ldots,y]=\mathbf{k}[x,s,t,y]=\mathbf{k}[x,s,t][I_{n}/x^{n}]=\mathbf{R}_{(n,S^{d}+T^{r}+XQ)}.
\]

\noindent For $(3)$, Lemma \ref{Lem:The-contraction-of-I_N}, assets
that $I_{nm+e_{0}}=\langle G,\, x^{e_{0}}F^{m},x^{n+e_{0}}F^{m-1},\ldots,x^{n(m-1)+e_{0}}F,x^{nm+e_{0}}\rangle_{\mathbf{k}[x,s,t]}$.
Therefore, 
\[
\mathbf{k}[x,s,t][\langle x^{nm+e_{0}}\rangle^{\mathbf{c}}/x^{nm+e_{0}}]=\mathbf{k}[x,s,t][G/x^{nm+e_{0}},\, x^{e_{0}}F^{m}/x^{nm+e_{0}},\ldots,x^{n(m-1)+e_{0}}F/x^{nm+e_{0}}].
\]
Thus, we get 
\[
\mathbf{k}[x,s,t][\langle x^{nm+e_{0}}\rangle^{\mathbf{c}}/x^{nm+e_{0}}]=\mathbf{k}[x,s,t][G/x^{nm+e_{0}},F/x^{n}]=\mathbf{k}[x,s,t][x^{e-e_{0}}z,y]=B_{(n,e_{0},Q)}.
\]
\end{proof}
\begin{rem}
Lemma \ref{Lem:The-contraction-of-I_N} and Lemma \ref{Lem:k[x,s,t][I_N/x_N]}
show that the contraction chain $\mathbf{k}[x,s,t]\supset\langle x\rangle^{\mathbf{c}}\supset\langle x^{2}\rangle^{\mathbf{c}}\supset\cdots\supset\langle x^{nm+e}\rangle^{\mathbf{c}}$,
associatrd to $B_{(n,e,Q)}$, consists of $nm+e$ distinct ideals
in $\mathbf{k}[x,s,t]$, while the induced exponential chain $\mathbf{k}[x,s,t]\subsetneq\mathbf{R}_{(1,S^{d}+T^{r}+XQ)}\subsetneq\cdots\subsetneq\mathbf{R}_{(n,S^{d}+T^{r}+XQ)}\subsetneq B_{(n,1,Q)}\subsetneq\cdots\subsetneq B_{(n,e,Q)}$
has only $n+e$ distinct (even non-isomorphic by virtue of Proposition
\ref{Prop:iso-preseve-<x>-and-n_1+e_1=00003Dn_2+e_2}) sub-algebras.
This will be a key observation to prove Proposition \ref{Prop:n_1=00003Dn_2-and-e_1=00003De_2}
and Theorem \ref{Thm:The-new-algebras-are-noniso-to-the-Russell},
that is, to distinguish $B_{(n,e,Q)}$; $e\neq0$ from Russell domains.
As we will see, the contraction chain of a Russell domain $\mathbf{R}_{(n',F)}$
consists of $n'$ distinct ideals in $\mathbf{k}[x,s,t]$, and the
exponential chain consists also of $n'$ non-isomorphic sub-algebras.
Therefore, in a sense, the number of non-isomorphic sub-algebras of
the exponential chain, represents a numeric characterization for these
$\mathbf{k}$-domains.
\end{rem}
The following corollary is a consequence of Theorem \ref{Thm: iso of new examples  preserves both chains}
and Lemma \ref{Lem:deg-of-LND}.
\begin{cor}
\label{Cor: the-exponential-chain invariant by LND} The exponential
chain $\mathbf{k}[x,s,t]\subset\mathbf{k}[x,s,t][\langle x\rangle^{\mathbf{c}}/x]\subset\cdots\subset\mathbf{k}[x,s,t][\langle x^{nm+e}\rangle^{\mathbf{c}}/x^{nm+e}]=B_{(n,e,Q)}$
of $B_{(n,e,Q)}$ is invariant by every locally nilpotent derivation
of $B_{(n,e,Q)}$. That is, every $\partial\in\mathrm{LND}(B_{(n,e,Q)})$
restricts to a locally nilpotent derivation of every member of the
exponential chain
\end{cor}

\subsection{Isomorphism classes and Automorphism groups}

\indent\newline\noindent  In the following proposition we give the
necessary conditions that $B_{(n_{1},e_{1},Q_{1})}$ and $B_{(n_{2},e_{2},Q_{2})}$,
where $n_{1}+e_{1}=n_{2}+e_{2}$, must satisfy to be isomorphic. This
will be done by comparing their exponential chains $\mathbf{k}[x,s,t]\subsetneq\mathbf{R}_{(1,S^{d}+T^{r}+XQ_{1})}\subsetneq\cdots\subsetneq\mathbf{R}_{(n_{1},S^{d}+T^{r}+XQ_{1})}\subsetneq B_{(n_{1},1,Q_{1})}\subsetneq\cdots\subsetneq B_{(n_{1},e_{1},Q_{1})}$
and $\mathbf{k}[x,s,t]\subsetneq\mathbf{R}_{(1,S^{d}+T^{r}+XQ_{2})}\subsetneq\cdots\subsetneq\mathbf{R}_{(n_{2},S^{d}+T^{r}+XQ_{2})}\subsetneq B_{(n_{2},1,Q_{2})}\subsetneq\cdots\subsetneq B_{(n_{2},e_{2},Q_{2})}$.
\begin{prop}
\label{Prop:n_1=00003Dn_2-and-e_1=00003De_2} Suppose that $B_{(n_{1},e_{1},Q_{1})}\simeq B_{(n_{2},e_{2},Q_{2})}$,
then $n_{1}=n_{2}$, and $e_{1}=e_{2}$.\end{prop}
\begin{proof}
Let $\Psi:B_{(n_{1},e_{1},Q_{1})}\longrightarrow B_{(n_{2},e_{2},Q_{2})}$
be a $\mathbf{k}$-isomorphism, and assume for contradiction that
$n_{1}<n_{2}$. By Theorem \ref{Thm: iso of new examples  preserves both chains}
and Lemma \ref{Lem:k[x,s,t][I_N/x_N]}, $\Psi$ restricts to a $\mathbf{k}$-isomorphism
between $\mathbf{k}[x,s,t][\langle x^{n_{1}}\rangle_{B_{(n_{1},e_{1},Q_{1})}}^{\mathbf{c}}/x^{n_{1}}]=\mathbf{R}_{(n_{1},S^{d}+T^{r}+XQ_{1})}$
and $\mathbf{k}[x,s,t][\langle x^{n_{1}}\rangle_{B_{(n_{2},e_{2},Q_{2})}}^{\mathbf{c}}/x^{n_{1}}]=\mathbf{R}_{(n_{1},S^{d}+T^{r}+XQ_{2})}$. 

\noindent Consider the sub-algebra $\mathbf{k}[x,s,t][\langle x^{n_{1}+1}\rangle_{B_{(n_{1},e_{1},Q_{1})}}^{\mathbf{c}}/x^{n_{1}+1}]$,
it coincides with $\mathbf{R}_{(n_{1},S^{d}+T^{r}+XQ_{1})}$ by virtue
of Lemma \ref{Lem:k[x,s,t][I_N/x_N]}. On the other hand, by Theorem
\ref{Thm: iso of new examples  preserves both chains} $\mathbf{k}[x,s,t][\langle x^{n_{1}+1}\rangle_{B_{(n_{1},e_{1},Q_{1})}}^{\mathbf{c}}/x^{n_{1}+1}]$
is isomorphic to $\mathbf{k}[x,s,t][\langle x^{n_{1}+1}\rangle_{B_{(n_{2},e_{2},Q_{2})}}^{\mathbf{c}}/x^{n_{1}+1}]$
which coincides with $\mathbf{R}_{(n_{1}+1,S^{d}+T^{r}+XQ_{2})}$.
However, the latter is not isomorphic to $\mathbf{R}_{(n_{1},S^{d}+T^{r}+XQ_{2})}$
by virtue of Corollary \ref{Cor:iso-e=00003D0-then-n_1=00003Dn_2-1},
a contradiction. Thus $n_{1}\geq n_{2}$ and by symmetry we deduce
that $n=n_{1}=n_{2}$. Since $n_{1}+e_{1}=n_{2}+e_{2}$ by virtue
of Proposition \ref{Prop:iso-preseve-<x>-and-n_1+e_1=00003Dn_2+e_2},
we get $e=e_{1}=e_{2}$, and we are done.
\end{proof}
Denote by $\mathrm{Iso}_{\mathbf{k}}\left(B_{(n_{1},e_{1},Q_{1})},B_{(n_{2},e_{2},Q_{2})}\right)$
the set of all $\mathbf{k}$-isomorphisms from $B_{(n_{1},e_{1},Q_{1})}$
to $B_{(n_{2},e_{2},Q_{2})}$. Proposition \ref{Prop:n_1=00003Dn_2-and-e_1=00003De_2}
implies that this set is empty whenever $(n_{1},e_{1})\neq(n_{2},e_{2})$.
The next proposition describes the set $\mathrm{Iso}_{\mathbf{k}}\left(B_{(n,e,Q_{1})},B_{(n,e,Q_{2})}\right)$
in terms of a sub-set of $\mathrm{Aut}_{\mathbf{k}}(\mathbf{k}[x,s,t])$
(the group of $\mathbf{k}$-automorphisms of $\mathbf{k}[x,s,t]$).
Let $\mathcal{A}$ be the sub-set of $\mathrm{Aut}_{\mathbf{k}}(\mathbf{k}[x,s,t])$
of automorphisms which preserve the ideal $\langle x\rangle_{\mathbf{k}[x,s,t]}$
and map $I=\langle x^{nm+e}\rangle_{B_{(n,e,Q_{1})}}^{\mathbf{c}}$
isomorphically to $J=\langle x^{nm+e}\rangle_{B_{(n,e,Q_{2})}}^{\mathbf{c}}$,
that is, 
\[
\mathcal{A}:=\{\psi\in\mathrm{Aut}_{\mathbf{k}}(\mathbf{k}[x,s,t]);\,\,\psi(x)=\lambda x\,;\,\lambda\in\mathbf{k}\setminus\{0\},\,\,\psi(I)=J\}.
\]
Then,
\begin{thm}
\label{Thm:iso-one-to-one-auto } There is a one-to-one correspondence
between the set $\mathrm{Iso}_{\mathbf{k}}\left(B_{(n,e,Q_{1})},B_{(n,e,Q_{2})}\right)$
and the set of $\mathbf{k}$-automorphisms $\mathcal{A}$.\end{thm}
\begin{proof}
Every $\mathbf{k}$-isomorphisms $\Psi:B_{(n,e,Q_{1})}\longrightarrow B_{(n,e,Q_{2})}$
restricts to $\Psi|_{\mathbf{k}[x,s,t]}$ a $\mathbf{k}$-automorphism
of the Derksen invariant $\mathbf{k}[x,s,t]$. On the other hand,
Proposition \ref{Prop:iso-preseve-<x>-and-n_1+e_1=00003Dn_2+e_2}
and Theorem \ref{Thm: iso of new examples  preserves both chains}
ensure that $\Psi$ preserves the ideal $\langle x\rangle_{\mathbf{k}[x,s,t]}$
and that $\Psi(I)=J$. Conversely, every $\mathbf{k}$-automorphism
$\psi$ of $\mathbf{k}[x,s,t]$ that preserves the ideal $\langle x\rangle_{\mathbf{k}[x,s,t]}$
and satisfies $\psi(I)=J$ extends, by virtue of Lemma \ref{Lem:extension-of-iso-to-algebraic-extension-of-degree-one},
in a unique way to $\widetilde{\psi}$ a $\mathbf{k}$-isomorphism
between $\mathbf{k}[x,s,t][I/x^{nm+e}]$ and $\mathbf{k}[x,s,t][J/x^{nm+e}]$.
These rings coincide with $B_{(n,e,Q_{1})}$ and $B_{(n,e,Q_{2})}$
by virtue of Lemma \ref{Lem:k[x,s,t][I_N/x_N]}. And we are done.
\end{proof}
The next corollary is a direct consequence of Theorem \ref{Thm:iso-one-to-one-auto }.
It describes the $\mathbf{k}$-automorphism group of $B_{(n,e,Q)}$
as a sub-group of the $\mathbf{k}$-automorphism group of the Derksen
invariant $\mathbf{k}[x,s,t]$. 
\begin{cor}
\label{Cor:aut-group-iso-to-A} The group $\mathrm{Aut}_{\mathbf{k}}(B_{(n,e,Q)})$
is isomorphic to the group $\mathcal{A}$ via the group isomorphism:
\[
\daleth:\mathrm{Aut}_{\mathbf{k}}(B_{(n,e,Q)})\overset{\sim}{\longrightarrow}\mathcal{A}\,\,;\,\,\,\daleth(\Psi)=\Psi|_{\mathbf{k}[x,s,t]}
\]
where $\Psi|_{\mathbf{k}[x,s,t]}$ is the restriction of $\Psi\in\mathrm{Aut}_{\mathbf{k}}(B_{(n,e,Q)})$
to the sub-algebra $\mathbf{k}[x,s,t]\subset B_{(n,e,Q)}$.
\end{cor}
Consider the exponential chain of $B_{(n,e,Q)}$ %
\footnote{Particular members of the exponential chain of $B_{(n,e,Q)}$ are
$\mathbf{k}[x,s,t]\hookrightarrow R_{(n,S^{d}+T^{r}+XQ)}\hookrightarrow B_{(n,e,Q)}$.
They correspond to $\mathrm{AL}_{0}(B_{(n,e,Q)})\hookrightarrow\mathrm{AL}_{\alpha}(B_{(n,e,Q)})\hookrightarrow\mathrm{AL}_{m\alpha}(B_{(n,e,Q)})$
for $\alpha=\min\{d,r\}$, see \cite[Section 2]{Alhajjar} for definitions
and some properties of $\mathrm{AL}_{i\in\mathbb{N}}$-invariants.%
}
\[
\mathbf{k}[x,s,t]\hookrightarrow\mathbf{R}_{(1,S^{d}+T^{r}+XQ)}\hookrightarrow\cdots\hookrightarrow\mathbf{R}_{(n,S^{d}+T^{r}+XQ)}\hookrightarrow B_{(n,1,Q)}\hookrightarrow\cdots\hookrightarrow B_{(n,e,Q)}.
\]
Every member of this chain represents an invariant sub-algebra of
$B_{(n,e,Q)}$, and we have the following.

\noindent {\footnotesize{} $\begin{array}{cccc}
 & \mathrm{Aut}_{\mathbf{k}}(\mathbf{R}_{(1,S^{d}+T^{r}+XQ)})\\
 & \cup\\
\mathrm{Aut}_{\mathbf{k}}(B_{(n,e,Q)})\subset\cdots\subset\mathrm{Aut}_{\mathbf{k}}(B_{(n,1,Q)})\subset\mathrm{Aut}_{\mathbf{k}}(\mathbf{R}_{(n,S^{d}+T^{r}+XQ)})\subset\ldots\subset & \mathrm{Aut}_{\mathbf{k}}(\mathbf{R}_{(2,S^{d}+T^{r}+XQ)}) & \subset & \mathrm{Aut}_{\mathbf{k}}(\mathbf{k}[x,s,t]),
\end{array}$}{\footnotesize \par}

\noindent {\footnotesize{} $\mathrm{LND}(B_{(n,e,Q)})=x\left(\mathrm{LND}(B_{(n,e-1,Q)})\right)=\cdots=x^{e}\left(\mathrm{LND}(\mathbf{R}_{(n,S^{d}+T^{r}+XQ)})\right)=x^{n+e}\left(\mathrm{LND}_{\mathbf{k}[x]}(\mathbf{k}[x,s,t])\right)$}
and 

\noindent {\footnotesize{} $\mathrm{LND}(\mathbf{R}_{(n,S^{d}+T^{r}+XQ)})=x\left(\mathrm{LND}(\mathbf{R}_{(n-1,S^{d}+T^{r}+XQ)})\right)=\cdots=x^{n-2}\left(\mathrm{LND}(\mathbf{R}_{(2,S^{d}+T^{r}+XQ)})\right)=x^{n}\left(\mathrm{LND}_{\mathbf{k}[x]}(\mathbf{k}[x,s,t])\right)$.}{\footnotesize \par}

\section{\textbf{New Exotic Structures on $\mathbb{C}^{3}$}}

Let $m,d,r\geq2$ be fixed such that $\gcd(d,r)=1$. For every $e\geq0$,
$n\geq1$ such that $(n,e)\neq(1,0)$, and every $Q\in\mathbf{k}[X,S,T]$,
we denote by $B_{(n,e,Q)}$ the following $\mathbf{k}$-domain: 
\[
B_{(n,e,Q)}:=\mathbf{k}[x,y,z,s,t]\simeq\mathbf{k}[X,Y,Z,S,T]/\langle X^{n}Y-S^{d}-T^{r}-X\, Q(X,S,T),\, Y^{m}-X^{e}Z-S\rangle.
\]

\begin{defn}
Recall that a smooth affine variety which is diffeomorphic to $\mathbb{R}^{2N}$
but not isomorphic to $\mathbb{C}^{N}$ is called an \emph{exotic}
$\mathbb{C}^{N}$.
\end{defn}

\subsection{A class of exotic threefolds}

\indent\newline\noindent  Let $\mathbf{k}=\mathbb{C}$ and assume
that $Q(0,0,0)\neq0$ and $e\gneq1$, then, by the Jacobian criterion,
the variety $V'=\mathrm{Sped}(B_{(n,e,Q)})$ is the smooth threefold
$x^{n}y-(y^{m}-x^{e}z)^{d}-t^{r}-xQ$ in $\mathbb{C}^{4}$, which
birationally dominates the affine space $V=\mathbb{C}^{3}$ under
the blowup morphism $\sigma_{I}:V'\longrightarrow V=\mathbb{C}^{3}$
; $\sigma_{I}(x,y,z,t)\mapsto(x,y^{m}-x^{e}z,t)$. The exceptional
divisor of the affine modification $\sigma_{I}:V'\longrightarrow V$,
see Proposition \ref{Prop:affine-mod}, coincides with $\mathrm{Spec}(A):=\{x=0\}\subset V'$
where $A:=\mathbb{C}[s,t,y,z]\simeq\mathbb{C}[S,T,Y,Z]/\langle S^{d}+T^{r},Y^{m}-S\rangle\simeq\mathbb{C}[T,Y,Z]/\langle Y^{md}+T^{r}\rangle$,
hence $\mathrm{Spec}(A)\simeq\mathbb{C}\times\Gamma_{md,r}$ where
$\Gamma_{md,r}=\mathrm{Spec}(\mathbb{C}[Y,T]/\langle Y^{md}+T^{r}\rangle)$.
Assume in addition that $\gcd(m,r)=1$. Since every irreducible singular
curve of the form $\Gamma_{N_{1},N_{2}}=\mathrm{Spec}(\mathbb{C}[Y,T]/\langle Y^{N_{1}}+T^{N_{2}}\rangle)$
where $\gcd(N_{1},N_{2})=1,\, N_{1}>N_{2}\geq2$, is contractible,
see \cite{Lin Zaidenberg}. We conclude that the necessary conditions,
see \cite[Proposition 4.2]{Zaidenberg}, for preserving the topology
under affine modifications are fulfilled. Therefore, by \cite[ Theorem 4.3]{Zaidenberg},
the variety $V'$ is contractible as a complex threefold, which yields
that $V'$ is diffeomorphic to $\mathbb{R}^{6}$ by virtue of the
Dimca-Ramanujam Theorem \cite[Theorem 3.2]{Zaidenberg}. Since $B_{(n,e,Q)}$
is not isomorphic to $\mathbb{C}^{[3]}$ by virtue of Theorem \ref{Thm:Derkson-invariant-of B}
or Corollary \ref{Cor:ML(B)=00003Dk[x]}, we deduce that $V'$ is
not isomorphic to the affine space $\mathbb{C}^{3}$. Therefore, $V'=\mathrm{Spec}(B_{(n,e,Q)})$
is an exotic $\mathbb{A}_{\mathbb{C}}^{3}$. 

Note that since $B_{(n,e,Q)}/\langle x\rangle\simeq\mathbf{k}[Y,Z,S,T]/\langle S^{d}+T^{r},Y^{m}-S\rangle\simeq\mathbf{k}[Y,Z,T]/\langle Y^{md}+T^{r}\rangle$,
the principle ideal $\langle x\rangle$ is prime whenever $\gcd(m,r)=1$.
On the other hand, $B_{(n,e,Q)}[x^{-1}]=\mathbf{k}[x^{-1},x,s,t]$
the localization of $B$ with respect to $x$, is a unique fraction
domain, therefore $B_{(n,e,Q)}$ is also a unique fraction domain
by virtue of \cite[Lemma 1]{Nagata}.

We put together the previous observations in the following.
\begin{thm}
\label{Thm:Spec(B)-is-an-exotic-three-fold} Under the conditions:
$($$\mathbf{k}=\mathbb{C}$, $\gcd(m,r)=1$, $e\geq2$, and $Q(0,0,0)\neq0$$)$.
The smooth factorial variety $\mathrm{Spec}(B_{(n,e,Q)})$ is diffeomorphic
to $\mathbb{R}^{6}$ but not isomorphic to $\mathbb{C}^{3}$. Hence,
$\mathrm{Spec}(B_{(n,e,Q)})$ is an exotic $\mathbb{C}^{3}$.
\end{thm}

\subsection{Comparing the class \textmd{$B_{(n,e,Q)}$} with Russell domains.}

\indent\newline\noindent  Here, we prove that domains of the form
$B_{(n,e,Q)}$; $e\neq0$ are not isomorphic to any of Russell $\mathbf{k}$-domains.

Denote by $\mathbf{R}_{(n',F)}$ the Russell $\mathbf{k}$-domain
corresponding to the pair $(n',F)$, that is, 
\[
\mathcal{\mathbf{R}}_{(n',F)}:=\mathbf{k}[x,s,t,y]\simeq\mathbf{k}[X,Y,S,T]/\langle X^{n'}Y-F(X,S,T)\rangle.
\]

\begin{thm}
\label{Thm:The-new-algebras-are-noniso-to-the-Russell} Suppose that
$B_{(n,e,Q)}\simeq\mathbf{R}_{(n',F)}$, then $e=0$ and $n=n'$.\end{thm}
\begin{proof}
Suppose that $B_{(n,e,Q)}\simeq\mathbf{R}_{(n',F)}$, then both rings
have the same Derksen and Makar-Limanov invariants. Therefore, by
Theorem \ref{Thm:Derkson-invariant-of B} and Corollary \ref{Cor:ML(B)=00003Dk[x]},
the Derksen and Makar-Limanov invariant of $\mathbf{R}_{(n',F)}$
is $\mathbf{k}[x,s,t]$ and $\mathbf{k}[x]$ respectively, where we
realize both $\mathbf{k}$-domains as sub-algebras of $B_{(n,e,Q)}[x^{-1}]=\mathbf{R}_{(n',F)}[x^{-1}]=\mathbf{k}[x^{-1},x,s,t]$.

\noindent  Let $\Psi:B_{(n,e,Q)}\longrightarrow\mathbf{R}_{(n',F)}$
be a $\mathbf{k}$-isomorphism between $B_{(n,e,Q)}$ and $\mathbf{R}_{(n',F)}$,
then it restricts to a $\mathbf{k}$-automorphism of the Makar-Limanov
invariant $\mathbf{k}[x]$. Hence, $\Psi(x)=\lambda x+c$ for some
$\lambda\in\mathbf{k}\backslash\{0\}$ and $c\in\mathbf{k}$. Therefore,
$\Psi$ induces $\overline{\Psi}$ an isomorphism between $B_{(n,e,Q)}/\langle x\rangle$
and $\mathbf{R}_{(n',F)}/\langle\lambda x+c\rangle$, which implies
that $c=0$. Indeed, assume that $c\neq0$, then $\mathbf{R}_{(n',F)}/\langle\lambda x+c\rangle\simeq\mathbf{k}[S,T]\simeq\mathbf{k}^{[2]}$.
On the other hand, $B_{(n,e,Q)}/\langle x\rangle\simeq\mathbf{k}[Y,T,Z]/\langle Y^{md}+T^{r}\rangle$
is either a non-domain (if $\gcd(m,r)\neq1$) or a semi-rigid $\mathbf{k}$-domain
with $\mathrm{ML}$-invariant equal to $\mathbf{k}[Y,T]/\langle Y^{md}-T^{r}\rangle$,
see \cite[Lemma 21]{Makar-Limanov: A new ring invariant}. Either
way $B_{(n,e,Q)}/\langle x\rangle$ is not isomorphic to $\mathbf{k}^{[2]}$
and hence the only possibility for $c$ is that $c=0$. Thus we have
$\Psi(x)=\lambda x$. Furthermore, since $\mathbf{R}_{(n',F)}/\langle x\rangle\simeq\mathbf{k}[S,T,Y]/P(S,T)$
where $P(S,T):=F(0,S,T)$, we can assume that $P(S,T)=S^{md}+T^{r}$.
Observe that $\mathbf{R}_{(n',F)}$ is the exponential modification
of $\mathbf{k}[x,s,t]$ with locos $(x^{n'},\langle x^{n'},F\rangle_{\mathbf{k}[x,s,t]})$
and exponential chain $\mathbf{k}[x,s,t]\subset\mathbf{R}_{(1,F)}\subset\cdots\subset\mathbf{R}_{(n',F)}$.

\noindent  The same argument, as in the proof of Proposition \ref{Prop:iso-preseve-<x>-and-n_1+e_1=00003Dn_2+e_2}
or Theorem \ref{Thm: iso-mod preservs bases and centers preserves contraction chains},
shows that $n+e=n'$. Assume for contradiction that $e\neq0$, then
$n<n'$. Theorem \ref{Thm: iso-mod preservs bases and centers preserves contraction chains}
asserts that $\Psi$ maps the contraction of $\langle x^{N}\rangle_{B_{(n,e,Q)}}$
isomorphically to the contraction of $\langle x^{N}\rangle_{\mathbf{R}_{(n',F)}}$,
and that $\Psi$ maps the sub-algebra $\mathbf{k}[x,s,t][\langle x^{N}\rangle_{B_{(n,e,Q)}}^{\mathbf{c}}/x^{N}]\subset B_{(n,e,Q)}$
isomorphically to the sub-algebra $\mathbf{k}[x,s,t][\langle x^{N}\rangle_{\mathbf{R}_{(n',F)}}^{\mathbf{c}}/x^{N}]\subset\mathbf{R}_{(n',F)}$,
for every $N\in\mathbb{N}$. In particular, $\Psi$ restricts to a
$\mathbf{k}$-isomorphism between $\mathbf{R}_{(n,S^{d}+T^{r}+XQ)}=\mathbf{k}[x,s,t][\langle x^{n}\rangle_{B_{(n,e,Q)}}^{\mathbf{c}}/x^{n}]\subset B_{(n,e,Q)}$
and $\mathbf{k}[x,s,t][\langle x^{n}\rangle_{\mathbf{R}_{(n',F)}}^{\mathbf{c}}/x^{n}]=\mathbf{R}_{(n,F)}\subset\mathbf{R}_{(n',F)}$.
Consider the sub-algebra $\mathbf{k}[x,s,t][\langle x^{n+1}\rangle_{B_{(n,e,Q)}}^{\mathbf{c}}/x^{n+1}]\subset B_{(n,e,Q)}$,
it is equal to $\mathbf{R}_{(n,S^{d}+T^{r}+XQ)}$ by virtue of Lemma
\ref{Lem:k[x,s,t][I_N/x_N]}. On the other hand, the sub-algebra $\mathbf{k}[x,s,t][\langle x^{n+1}\rangle_{\mathbf{R}_{(n',F)}}^{\mathbf{c}}/x^{n+1}]=\mathbf{R}_{(n+1,F)}$
is not equal (even non-isomorphic) to $\mathbf{R}_{(n,F)}$, a contradiction.
Therefore, the only possibility is $e=0$ and $n=n'$, as desired.
\end{proof}
As a consequence of Theorem \ref{Thm:The-new-algebras-are-noniso-to-the-Russell},
we have the following.
\begin{cor}
Under the conditions: $($$\mathbf{k}=\mathbb{C}$, $\gcd(m,r)=1$,
$e\geq2$, and $Q(0,0,0)\neq0$$)$. The variety $\mathrm{Spec}(B_{(n,e,Q)})$
is not isomorphic to $\mathrm{Spec}(\mathbf{R}_{(n',F)})$. Consequently,
\textup{$\mathrm{Spec}(B_{(n,e,Q)})$} represents a new exotic $\mathbb{C}^{3}$.
\end{cor}

\section*{\textbf{\normalsize Acknowledgments}}

\thanks{I would like to thank all my teachers at the Department of Mathematics
of the Damascus University, for doing a lot of corrections and many
stimulating discussions. This research is supported by a grant from
Syria's Ministry of Higher Education.}

\indent\newline

\end{document}